\documentclass[12pt]{article}

\usepackage[titletoc,toc,title]{appendix}                                                                                 \usepackage{manfnt}  
\usepackage[utf8]{inputenc}
\usepackage{graphicx,amsmath,amsfonts,amsthm}
\usepackage[T1]{fontenc}
\usepackage{listings}
\usepackage{mathrsfs}
\usepackage{thmtools}
\usepackage{amssymb,enumitem} 
\usepackage{comment}
\usepackage[all]{xy}
\usepackage{tikz}
\usepackage{tikz-cd}
\usetikzlibrary{patterns,decorations.pathreplacing,babel,decorations.markings,shapes,positioning,intersections,quotes,arrows}
\usepackage{faktor}
\usepackage{xfrac}
\usepackage{afterpage}
\usepackage{hyperref}
\usepackage{xspace}

\numberwithin{equation}{section}

\newtheorem{theorem}{Theorem}[section]
\newtheorem{lemma}[theorem]{Lemma}
\newtheorem{proposition}[theorem]{Proposition}
\newtheorem{corollary}[theorem]{Corollary}

\newtheorem{sinnada}[theorem]{}

\theoremstyle{definition}
\newtheorem{definition}[theorem]{Definition}
\newtheorem{definitions}[theorem]{Definitions}

\newtheorem{notation}[theorem]{Notation}
\newtheorem{comentario}[theorem]{Comment}
\newtheorem{remark}[theorem]{Remark}

\swapnumbers

\newcommand{\az}[1]{\textcolor{blue}{#1}}

\newcommand{\rjj}{\color{red}}

\theoremstyle{remark}

\newcommand \Ho {\mathcal{H}o(\mathscr{C})}
\newcommand \HAo {\mathcal{H}o(\mathscr{A})}
\newcommand \Hof {\mathcal{H}o^{f}(\mathscr{C})}
\newcommand \HAof {\mathcal{H}o^{f}(\mathscr{A})}
\newcommand \Hol {\mathcal{H}o^{\ell}(\mathscr{C})}
\newcommand \fcHo {\mathcal{H}o_{fc}(\mathscr{C})}
\newcommand \Hofc {\mathcal{H}o(\mathscr{C}_{fc})}
\newcommand \Holf {\mathcal{H}o^{\ell f}(\mathscr{C})}
\newcommand \Horc {\mathcal{H}o^{r c}(\mathscr{C})}

\newcommand \fclfHo {\mathcal{H}o_{fc}^{\ell f}(\mathscr{C})}
\newcommand \fcrcHo {\mathcal{H}o_{fc}^{\ell f}(\mathscr{C})}

\newcommand{\catHofc}[1]{\mathcal{H}o(\mathscr{#1}_{fc})}

\newcommand*\widebar[1]{%
   \hbox{%
     \vbox{%
       \hrule height 0.5pt %
       \kern0.5ex
       \hbox{%
         \kern-0.1em
         \ensuremath{#1}%
         \kern-0.1em%
       }%
     }%
   }%
}

\newdir{ >}{{}*!/-6pt/@{>}}

\newcommand{\germ}{\stackrel{ge}{\thicksim}}
\newcommand{\crpd}{\stackrel{cd}{\thicksim}}

\newcommand{\hpy}
           {
            \hspace{-1ex}
            \xymatrix{ {} \ar@2{~>}@<0.25ex>[r] & {} }
            \hspace{-0.6ex}
            }

\newcommand{\Hpy}[1]  
    {
    \hspace{-0.6ex}
    \xymatrix@C=3ex{{} \ar@2{~>}[r]^{#1} & {}}
    \hspace{-0.6ex}
    }

\newcommand{\mre}[1]  
    {\xymatrix@1{{} \ar[r]^{#1}| \circ & {}}}

\newcommand{\mrf}[1]  
    {\xymatrix@1{{} \ar@{->>}[r]^{#1} & {}}}

\newcommand{\mrfe}[1] 
    {\xymatrix@1{{} \ar@{->>}[r]^{#1}|\circ & {}}}
  
\newcommand{\mrc}[1] 
    {\xymatrix@1{{} \hspace{1ex} \ar@{>->}[r]^{#1} & {}}}

\newcommand{\mrce}[1] 
    {\xymatrix@1{{} \hspace{1ex} \ar@{>->}[r]^{#1}|\circ & {}}}
    
\newcommand{\mr}[1]{\stackrel{#1}{\longrightarrow}}

\newcommand{\Mr}[1]{\stackrel{#1}{\Longrightarrow}}
\newcommand{\Ml}[1]{\stackrel{#1}{\Longleftarrow}}

\newcommand{\mvr}[1]{\xymatrix{{} \ar@{~>}[r]^{#1} & {}}}  

\newcommand{\mrdos}[2]
  {
   \xymatrix
      {
       {} \ar@<5pt>[r]^{#1} \ar@<-4pt>[r]^{#2} & {}
      }
  } 

\newcommand{\Mrr}[2] 
  {
   \xymatrix@C=#1ex
       { 
        {} \ar@{=>}[r]^{#2} & {}
       }
  }
  
\newcommand{\cellrdE}[3] 
{\xymatrix@C=7ex@R=2.4ex
        {
		\ar@<1.9ex>[r]^{#1} 
		\ar@{}@<-1.3ex>[r]^{\!\! {#2} \, \!\Downarrow}
		\ar@<-1.1ex>[r]_{#3} & 
		}
}
\newcommand{\cellrdEcorta}[3] 
{\xymatrix@C=5ex@R=2.4ex
       {
		\ar@<1.6ex>[r]^{#1} 
		\ar@{}@<-1.3ex>[r]^{\!\! {#2} \, \!\Downarrow}
		\ar@<-1.1ex>[r]_{#3} & 
	   }
}
\newcommand{\cellrd}[3] 
 {
  \xymatrix@C=5ex@R=2.4ex
         {
          {} \ar@<1.1ex>[r]^{#1}
             \ar@{}@<-1.3ex>[r]^{\Downarrow \; {#2}}
             \ar@<-1.2ex>[r]_{#3}
          & {}
         }
}
\newcommand{\cellrdb}[3] 
 {
  \xymatrix@C=7ex@R=2.4ex
         {
          {} \ar@<1.4ex>[r]^{#1}
             \ar@{}@<-1.3ex>[r]^{\Downarrow \; {#2}}
             \ar@<-1.1ex>[r]_{#3}
          & {}
         }         
 }
 \newcommand{\scellrd}[3] 
 {
  \xymatrix@C=4.5ex@R=2.4ex
         {
          {} \ar@<1.4ex>[r]^{#1}
             \ar@{}@<-1.3ex>[r]^{\!\! \Downarrow \, {#2}}
             \ar@<-1.1ex>[r]_{#3}
          & {}
         }
}
          
\newcommand{\modif}[3] 
 {
  \xymatrix@C=7ex@R=2.4ex
         {
          {} \ar@<1.6ex>@{=>}[r]^{#1}
             \ar@{}@<-1.3ex>@{=>}[r]^{\!\! {#2} \, \!\downarrow}
             \ar@{}@<-1.1ex>[r]_{#3}
          & {}
         }
 }

\newcommand{\cellld}[3] 
 {
  \xymatrix@C=6ex@R=2.4ex
         {
            {}
          & {} \ar@<1.0ex>[l]^{#3}
          \ar@{}@<-1.7ex>[l]^{\!\! {#2} \, \!\Downarrow}
	                                 \ar@<-1.7ex>[l]_{#1}
         }
 }

\newcommand{\cellpairrd}[4] 
 {
  \xymatrix@C=8ex@R=2.2ex
         {
          {} \ar@<1.6ex>[r]^{#1}
             \ar@{}@<-1.3ex>[r]^{\!\! \Downarrow \, {#2} 
                                 \;\;\; \Downarrow \, {#3} }
             \ar@<-1.1ex>[r]_{#4}
          & {}
         }
 }
 
\newcommand{\cqd}{\hfill$\Box$}

\newcommand{\cc}{\mathcal} 
\newcommand{\sr}{\mathscr}
\newcommand{\ol}{\overline}

\begin{document}

\title{The 2-localization of a model category}

\author{Eduardo J. Dubuc, Jaqueline Girabel}
  
\date{\vspace{-5ex}}

\maketitle



\begin{abstract} 
In this paper we elaborate on a 2-categorical construction of the homotopy category of a Quillen model category.
Given any  category $\sr{A}$ and a class of morphisms 
$\Sigma \subset \sr{A}$ containing the identities, we construct a 2-category $\HAo$ obtained by the addition of 2-cells determined by homotopies. A salient feature here is the use of a novel notion of cylinder introduced in \cite{e.d.2}. The inclusion 2-functor $\sr{A} \mr{} \HAo$ has a universal property which  yields the 2-localization of $\sr{A}$ at $\Sigma$ provided that the arrows of $\Sigma$ become equivalences in $\HAo$. This result together with a fibrant-cofibrant replacement is then used to obtain the 2-localization of a model category $\sr{C}$ at the weak equivalences  $\cc{W}$. The set of connected components of the hom categories yields a novel proof of Quillen's results. We follow the general lines established in \cite{e.d.2}, \cite{e.d.} for model bicategories.

  
\end{abstract} 





\tableofcontents



\vspace{4ex}

{\bf \Large Introduction}

\vspace{1ex} 

In this paper we study a 2-dimensional version of  Quillen's homotopy category construction. As in Quillen's construction, our input is a 1-model category, but our output are (2,1)-categories, and we set and establish their localization universal property in the 3-category of 2-categories and 2-functors. Of particular interest will be to study the relation with the simplicial localization developed  in \cite{DK1}, \cite{DK2}, and the possibility of a construction of this localization  using homotopy-like constructions in place of hammocks. We follow the general lines established in \cite{e.d.2}, \cite{e.d.} where the notion of model 2-category is introduced, and a fully 2-dimensional version is developed, with input and output (2,2)-categories, and the localization  universal property set in the 3-category of 2-categories and pseudofunctors. The developments here are not just an adaptation of the general theory to a particular case, in particular most of proofs produced are not simplified versions of the ones of the $(2,\,2)$ case.  There are two substantial differences: First, the absence of non-invertible 2-cells not only simplifies the theory, but also avoids the need for the new axioms and definitions introduced in loc. cit. above, which conform a different theory in form and spirit. Second, we introduce the use of functorial factorizations in the fibrant-cofibrant replacement, which allows us to dispense the use of pseudofunctors keeping only 2-functors, a fact that adds important simplifications but which has its own shortcomings. Only left homotopies suffice, there is no need to consider right homotopies. There are no fibrant nor cofibrant replacement 2-functors at the level of the homotopy \mbox{2-categories,} but a simultaneous fibrant-cofibrant replacement 2-functor. 

We see in an appendix the role played by the right homotopies, and how the original Quillen factorization axiom suffices to determine a fibrant and a cofibrant replacement at the homotopy level, but given by pseudofunctors, not 2-functors, and the same happens with the localization arrow, 
 also given by a pseudofunctor.  

\vspace{2ex}    

We pass now to describe the content of the paper:

\vspace{1ex}

{ \bf Section \ref{prelims}.} In a first section we fix notation, terminology, and recall some definitions whose explicit formulation become necessary to formulate the statements in this paper and develop their proofs. 

In particular, given a 2-category $\sr{A}$, a class of morphisms 
$\Sigma \subset \sr{A}$ containing the identities, and any 2-category $\sr{D}$, we denote by $Hom_{p}(\sr{A}, \sr{D})$ the
2-category with objects the 2-functors,  arrows the pseudonatural transformations and  2-cells the modifications, and by 
$Hom_{p}(\sr{A}, \sr{D})_{+}$ the full 2-subcategory spanned by the 2-functors which send the class $\Sigma$ into equivalences. Recall that we can consider categories as trivial 2-categories whose only 2-cells are the identities.
 
\vspace{1ex}

{ \bf Section \ref{hacheo}.} Given a category $\sr{A}$ and a class of morphisms 
$\Sigma \subset \sr{A}$ containing the identities, this section concerns the construction of a 2-category $\HAo$ obtained by the addition of 2-cells to 
$\sr{A}$, and such that the inclusion 2-functor $\sr{A} \mr{i} \HAo$ has the following universal property:

\vspace{1ex}

{\bf (up)} \emph{Precomposition with $i$, 
\mbox{$i^*: Hom_{p}(\HAo, \mathscr{D})_+\longrightarrow 
Hom_{p}(\mathscr{A}, \mathscr{D})_+$} induces an isomorphism of 2-categories for any 2-category $\sr{D}$.}

\vspace{1ex}

It is clear that $\HAo$ will be the 2-localization of $\sr{A}$ at $\Sigma$ as soon as the arrows of $\Sigma$ became equivalences in $\HAo$. We show that a sufficient condition for this is that 
$\Sigma$ has the 3 x 2 property and is generated by sections and retractions.

We construct $\HAo$ using cylinders and homotopies, so we call the members of $\Sigma$ \emph{weak equivalences} and $\HAo$ the 
\emph{homotopy 2-category}. Homotopies will determine the 2-cells of 
$\HAo$. We work using a definition of cylinder with respect to a class $\Sigma$, and the corresponding homotopies, introduced in 
\cite[3.2.3]{S}, which is more general and less rigid than Quillen's. 

 A \emph{cylinder} $C$ 
for an object $X$ 
is a configuration
%
%
        $\xymatrix@R=2ex{
     X \ar@<3pt>[rr]^{d_0}
        \ar@<-3pt>[rr]_{d_1} 
        \ar[dr]_{x} && W,
        \ar[dl]^{s}| \circ\\
        & Z \\ }$
    where $s \in \Sigma$ and $sd_0=sd_1=x$ \footnote{In \cite[3.3.1]{S} it is also introduced a more restricted notion under the name  "\emph{fork}", where $x$ is assumed to be a weak equivalence.}.
A
\emph{homotopy} $H$ with cylinder $C$ from an arrow $X \mr{f} Y$ to an arrow \mbox{$X \mr{g} Y$} is defined in the usual way as an arrow $W \mr{h} Y$ such that $f = h \, d_0$ and    
$g = h \, d_1$.
 
Homotopies can be vertically (actually sequences of composable homotopies) and horizontally composed but they are not yet the 2-cells of a 
\mbox{2-category} $\HAo$ with the same objects and arrows as $\cc{C}$. The  2-category axioms are equations, thus, the homotopy on the left of the equation should coincide with the one on the right, and if not, they should be identified. We do this by defining an equivalence relation as follows:
 
If $\sr{A}$ is already a 2-category and $s$ is an equivalence, it follows that there is a unique 2-cell $\widehat{C}: d_0 \Mr{} d_1$ such that 
$s\, \widehat{C} = x$, and this determines a 2-cell 
$\widehat{H} = h\, \widehat{C} : f \Mr{} g$. A 2-functor 
$\sr{A} \mr{F} \sr{D}$ into a 2-category $\sr{D}$ sends cylinders to cylinders and homotopies to homotopies, thus if it sends weak equivalences into equivalences, there is a 2-cell $\widehat{FH}: Ff \Mr{} Fg$. 

In view of this, given such a 2-functor $F: \sr{A} \longrightarrow \mathscr{D}$, we could extend $F$ to a 2-functor $\widetilde{F}$ defined in $\HAo$ as in the diagram
$\xymatrix@R=2ex{
        \ar @{} [drr] |(.45){}
        \sr{A} \hspace{1mm} \ar@{^{(}->}[rr]^{i} \ar[dr]_{F} && \HAo \ar@{-->}[dl]^(.5){\exists!\widetilde{F}} \\
        & \mathscr{D} & \\
    }
$
by defining $\widetilde{F} = F$ on objects and arrows, and 
$\widetilde{F}(H) = \widehat{FH}$ on homotopies. For $\widetilde{F}$ to be well defined on equivalence classes it is then necessary that for any two equivalent homotopies $H, \,K$,  
 the equality $\widehat{FH} = \widehat{FK}$ must hold for all such 2-functors. We just set this condition as the definition of the equivalence relation we use to define 
$\HAo$, and call it the \emph{ad-hoc relation}.

We show that this relation indeed determines a 2-category, which is the 2-category sought, that is, it satisfies the universal property {\bf (up)} above.

\vspace{1ex}

We consider also another relation between homotopies, that we call the \emph{germ relation}. There is a natural definition of morphisms of cylinders, which then form a category. Fixing two objects and two arrows 
$X \mrdos{f}{g} Y$, the set of homotopies $Hpy(X, \,Y)(f,\,g)(C)$ with cylinder $C$ defines a contravariant set valued functor on the variable $C$, and the germ relation is the equivalence relation which computes its colimit. This relation was already considered by Quillen himself in \cite{Qui} and we think it will be meaningful for the work in higher dimensions. Here we show that it is finer than the ad-hoc relation, a fact that comes in handy to establish equality of 2-cells in $\HAo$, as we see in the next section. 

\vspace{1ex}   

{\bf Section \ref{Loc_model}.} In this section we apply the construction of the homotopy 2-category in Section \ref{hacheo} to develop a 2-dimensional version of Quillen's localization of a model category 
$\sr{C}$, $(\mathcal{F}, \; \mathcal{C}o\!f, \; \mathcal{W})$, notation refers to \ref{MC}. We set $\Sigma = \cc{W}$ and we will consider the cases  $\sr{A} = \sr{C}$ and $\sr{A} = \sr{C}_{fc}$, the full subcategory of fibrant-cofibrant objects. 

\vspace{1ex}

In a first subsection \ref{locationC_fc} we work with $\sr{C}_{fc}$. 
As usual, we check that weak equivalences between fibrant-cofibrant objects factor though a fibrant-cofibrant object as a section followed by a retraction, both weak equivalences. Then using Theorem \ref{2-loc-W-split} 
we establish that the 2-functor \mbox{$\sr{C}_{fc} \mr{i} \Hofc$} is the 
2-localization in a strict sense of $\sr{C}_{fc}$ at the weak equivalences. 
For any 2-category $\sr{D}$, precomposition with $i$ establishes a 2-category isomorphism   
$$
i^*: Hom_p(\cc{H}o(\sr{C}_{fc}), \mathscr{D}) \mr{}
             Hom_{p}(\sr{C}_{fc}, \mathscr{D})_+ 
$$
Note that 
$\sr{C}_{fc}$ is not a model category. 

There is another 2-category $\fcHo \subset \Ho$, the full 2-subcategory of 
$\Ho$ whose objects are the fibrant-cofibrant objects, which is coarser than  
$\Hofc$ since for the latter the ad-hoc relation may include \mbox{2-functors} on 
$\sr{C}_{fc}$ which do not extend to $\sr{C}$. We abound on this in Remark \ref{listo}.

\vspace{1ex}

In a second subsection \ref{Quillen_2cells} we recall Quillen's cylinders, denoted here \mbox{$q$-cylinders}, as cylinders
$\xymatrix@R=2ex{
     X \ar@<3pt>[rr]^{d_0}
        \ar@<-3pt>[rr]_{d_1}
        \ar[dr]_{x} && W,
        \ar[dl]^{s}| \circ\\
        & Z \\ }$
with $Z = X$, $x = id_X$, that is
$\xymatrix@R=2ex{
     X \ar@<3pt>[rr]^{d_0}
        \ar@<-3pt>[rr]_{d_1}
        \ar[dr]_{id} && W,
        \ar[dl]^{s}| \circ\\
        & X \\ }$
where furthermore it is required that 
$X \amalg X \mr{\binom{d_0}{d_1}} W$ be a cofibration. Quillen's left homotopies, denoted here  \mbox{$q$-homotopies}, are homotopies whose cylinder is a $q$-cylinder. 

We show using the germ relation that in $\Hofc$  any 2-cell $[H]$ can be given with  $H$ a $q$-homotopy. On the other hand, we show how 2-cells determined by $q$-homotopies compose vertically. In this way, the 2-category $\Hofc$ coincides with the 2-category with 2-cells given by classes (not sequences) of $q$-homotopies. We comment that general homotopies are necessary for the horizontal composition because the composite of a $q$-homotopy $H$ with an arrow $l$ on the right, $H\,l$, is given by a cylinder which is not any more a $q$-cylinder, see Comment \ref{rcompoq}. 

\vspace{1ex}

Taking the connected components in the hom-categories of our construction we derive Quillen's homotopy category of $\sr{C}_{fc}$.

\vspace{1ex}

In  subsection \ref{2_loc_srC} we set $\sr{A} = \sr{C}$ and develop its localization \mbox{$\mathscr{C} \mr{q} \Hofc $}. Our main theorem is the following: For any 2-category $\sr{D}$, precomposition with $q$ establishes a 2-category pseudoequivalence   
$$
q^*: Hom_p(\cc{H}o(\sr{C}_{fc}), \mathscr{D}) \mr{}
             Hom_{p}(\mathscr{C}, \mathscr{D})_+ 
$$
We remark that we will need the fundamental theorem \ref{hom_iso}, not only the localization theorem \ref{2-loc-W-split} as it was in the case 
$\sr{A} = \sr{C}_{fc}$. We construct a fibrant and a cofibrant replacement which, while they are defined on arrows, do not necessarily preserve  composition. At this point we assume that the factorization axiom is functorial, a requirement frequently used in the post-Quillen literature, and show that this yields a fibrant-cofibrant replacement functor  
\mbox{$\mathscr{C} \mr{r} \mathscr{C}_{fc}$}. We take then  the composition with \mbox{$\mathscr{C}_{fc} \mr{i} \Hofc$} and prove that it establishes the localization $q = i\,r$, \mbox{$\mathscr{C} \mr{q} \Hofc $}  sought. 
We also remark here that for the proof developed in subsection \ref{2location} it is necessary to count with a fibrant-cofibrant replacement 2-functor 
$\Ho \mr{\ol{r}} \Hofc$ at the level of the homotopy categories. In subsection \ref{fibrant-cofibrant} we construct such a 2 functor.  With this, the proof goes as follows: We show that $i\,r = \ol{r}\,i$,
%
%
so that $q = \ol{r}\,i$, and we prove that the 2-functor $\ol{r}^\ast$ of precomposition with  $\ol{r}$ is a pseudoequivalence of 2-categories. Since by Theorem \ref{hom_iso} we already know that this is the case for the 2-functor $i^*$, the result follows.

Note that the 2-functor $\ol{r}$ is a simultaneous fibrant-cofibrant replacement, not a composition of a fibrant with a cofibrant replacement 2-functors, which do not necessarily exist.
The reader should be aware also that since in subsection \ref{2_loc_srC} we are assuming functorial factorization, the 2-categories $\Hofc$ and $\fcHo$ coincide, discarding a possible inconsistency. This fact is not used to prove the theorem. 

\vspace{1ex}

{\bf Appendix.} In this appendix we complete the picture. We briefly describe what is the situation of the problem of the 2-localization of a model category in the absence of functorial factorizations.  Complete definitions and proofs are developed in [DDS2] for model bicategories. The interest in the particular case of model categories is twofold. One is that it is much simpler, avoiding the difficulties, complications and new axioms necessary to be able to deal with non invertible 2-cells. The other is that surprisingly the resulting theory is different from the theory developed in section 3, and this difference adds  a new light on the functorial factorization axiom and its consequences. 

The concepts dual to that of cylinder and homotopy become necessary, left fibrant and right cofibrant homotopies play an essential role.
Pseudofunctors are unavoidable, the localizing arrow \mbox{$q:\mathscr{C} \longrightarrow \fcHo$} is a pseudofunctor, and the localization theorem becomes:  
\emph{Precomposition with $q$ establishes a 2-category pseudoequivalence}
$$
q^*: pHom_p(\fcHo, \mathscr{D}) \mr{}
             pHom_{p}(\mathscr{C}, \mathscr{D})_+
$$
\emph{for every 2-category $\mathscr{D}$.} Now the hom 2-categories are $pHom_{p}$, that is the 2-category whose objects are the pseudofunctors. More importantly,  there are fibrant and cofibrant replacements at the level of the homotopy 2-categories, which allow to obtain the fibrant-cofibrant replacement as a composition. There is no pseudofunctor \mbox{$\mathscr{C} \mr{} \Hofc$}, now the homotopy category is $\cc{H}o_{fc}(\sr{C})$, not $\cc{H}o(\sr{C}_{fc})$, which is a different 2-cate,gory in the absence of functorial factorization.

  \section{Preliminaries} \label{prelims}
This section is necessary to recall some concepts that we will use explicitly later, and to fix the notation and terminology.
\subsection{2-Categories}

\begin{notation} \label{notation}
A 2-category has \emph{objects} $X$, \emph{arrows} $X \mr{f} Y$ and \emph{2-cells} $f \Mr{\alpha} g$ (we will also say \emph{morphisms}, \emph{maps},  or
\emph{1-cells} to refer to the arrows).
The arrows and 2-cells are composed horizontally, which we simply denote as a juxtaposition, and the 2-cells are composed vertically, which we denote by $\,\circ\,$, the arrows are considered to be identity \mbox{2-cells } for vertical composition, we will denote both $f$ and $id_f$ as appropriate in the context. Thus, given $f \Mr{\alpha} g$, we have $\alpha \circ f = \alpha$ and $g \circ \alpha = \alpha$.

\end{notation}

\begin{remark} \label{horizontal_comp}
    In the presence of vertical composition, to determine a horizontal composition of 2-cells it is enough to define a horizontal composition between 2-cells and morphisms, which some authors call \emph{wishkering}:
    
        For each \begin{tikzcd}
    X \arrow["l"]{r}
    & Y \arrow[shift left=10pt, "f"]{r}[name=LUU, below]{}
    \arrow[shift right=10pt, "g"']{r}[name=LDD]{}
    \arrow[Rightarrow,to path=(LUU) -- (LDD)\tikztonodes]{r}{\alpha}
    & \hspace{1mm} Z \arrow["r"]{r}
    &W,
    \end{tikzcd}
    suppose we have defined 2-cells $r\,\alpha$ and 
    $\alpha \,l$. If the following axioms hold
    \begin{enumerate}
        \item For each $\xymatrix{X \ar[r]^f & Y\ar[r]^g & Z}, Id_gf=gId_f=Id_{gf}$. 
        \item For each \begin{tikzcd}
                  X \arrow["l"]{r}
                  & Y \arrow[shift left=20pt, "\hspace{6mm}f"]{r}[name=LUU, below]{} \arrow["\hspace{6mm} g"']{r}[name= LUD]{}
                  \arrow[swap]{r}[name=LDU]{}
                  \arrow[shift right=20pt, "\hspace{6mm}k"']{r}[name=LDD]{}
                  \arrow[Rightarrow,to path=(LUU) -- (LUD)\tikztonodes]{r}{\hspace{-6mm} \alpha}
                  \arrow[Rightarrow,to path=(LDU) -- (LDD)\tikztonodes]{r}{\hspace{-6mm} \beta}
                  & Z \arrow["r"]{r}
                  &W
            \end{tikzcd},  $(\beta l) \circ (\alpha l)=(\beta \circ \alpha)l$, \mbox{$(r\alpha) \circ (r\beta)=r(\beta \circ \alpha)$.}
        \item For each
            \begin{tikzcd}
                X \arrow[shift left=10pt, "f"]{r}[name=LUU, below]{}
                \arrow[shift right=10pt, "g"']{r}[name=LDD]{}
                \arrow[Rightarrow,to path=(LUU) -- (LDD)\tikztonodes]{r}{\alpha}
                & \hspace{1mm} Y \arrow[shift left=10pt, "f'"]{r}[name=RUU, below]{}
                \arrow[shift right=10pt, "g'"']{r}[name=RDD]{}
                \arrow[Rightarrow,to path=(RUU) -- (RDD)\tikztonodes]{r}{\alpha'}
                &Z,
            \end{tikzcd}
    $(g'\alpha) \circ (\alpha' f) = (\alpha' g) \circ (f'\alpha)$
    \end{enumerate}
    then the horizontal composition of any composable pair 
    $\alpha',\: \alpha$, can be defined by
    \mbox{$\alpha' \, \alpha = (g'\alpha) \circ (\alpha' f) = (\alpha' g) \circ (f'\alpha).$}
\end{remark}

\begin{definitions} \label{def_pseudo}
    A \emph{2-functor} $F:\mathscr{C} \longrightarrow \mathscr{D}$ between 2-categories sends objects of $\mathscr{C}$ into objects of $\mathscr{D}$, arrows of $\mathscr{C}$ into arrows of $\mathscr{D}$ and 2-cells of $\mathscr{C}$ into 2-cells of $\mathscr{D}$, preserving all the structure,  vertical and horizontal compositions, and identities.
 
 \vspace{1ex}
 
 If $G$ is another 2-functor between the same 2-categories, a \emph{pseudo-natural transformation} $\eta: F \Longrightarrow G$ consists of a family of arrows
    \mbox{$\eta_X:FX \longrightarrow GX$ in $\mathscr{D}$,} for each $X$ in $\mathscr{C}$, and for each arrow $f: X \longrightarrow Y$ in 
    $\mathscr{C}$ an invertible 2-cell $\eta_f: Gf\eta_X \Longrightarrow \eta_YFf$ in $\mathscr{D}$
     
    \begin{center}
        $\xymatrix{
            FX \ar[r]^{\eta_X} \ar[d]_{Ff} & GX \ar[d]^{Gf} \ar@{}[dl]|{\rotatebox[origin=c]{45 }{$\Leftarrow$} \eta_f}\\
            FY \ar[r]_{\eta_Y} & GY}$
    \end{center}
such that for each 2-cell $\alpha: f \Longrightarrow g$ it is verified:
    \begin{enumerate}
        \item For each $X$ in $\mathscr{C}$, $\eta_{id_X}=\eta_X$;
        \item Given $\xymatrix{ X \ar[r]^{f} & Y \ar[r]^{g} & Z }$ in $\mathscr{C}$, we have $\eta_g F(f )\circ G(g) \eta_f=\eta_{gf}$, according to notation \ref{notation}.
        \begin{center}
        $\xymatrix{
            FX \ar[r]^{\eta_X} \ar[d]_{Ff} & GX \ar[d]^{Gf} \ar@{}[dl]|{\rotatebox[origin=c]{45 }{$\Leftarrow$} \eta_f}\\
            FY \ar[r]_{\eta_Y} \ar[d]_{Fg} & GY \ar[d]^{Gg}\ar@{}[dl]|{\rotatebox[origin=c]{45 }{$\Leftarrow$} \eta_g} & =\hspace{3mm} \\
            FZ \ar[r]_{\eta_Z} & GZ\\}$
            $\xymatrix{
            FX \ar[r]^{\eta_X} \ar[dd]_{Fgf} & GX \ar[dd]^{Ggf} \ar@{}[ddl]|{\rotatebox[origin=c]{45 }{$\Leftarrow$} \eta_{gf}}\\
            \\
            FZ \ar[r]_{\eta_Z} & GZ}$
        \end{center}
        \item For each 2-cell $\alpha: f \Longrightarrow g : X \longrightarrow Y$ in $\mathscr{C}$, the equation $\eta_g \circ (G\alpha\,\eta_X) holds =
                        (\eta_YF\alpha) \circ \eta_f:$
$$
\xymatrix@R=3ex@C=4ex
    {
     FX \ar[rr]^{\eta_X}
        \ar[dd]_{Fg}
  && GX \ar@<-10pt>[dd]_{Gf}
        \ar@<10pt>[dd]^{Gg \hspace{4ex} = \hspace{3ex}}
        \ar@{}[dd]|{\underset{\Longrightarrow}{G\alpha}}
        \ar@{}[ddll]|{\rotatebox[origin=c]{45}{$\Leftarrow$}
                                                   \eta_{g}}
\\ \\
     FY \ar[rr]_{\eta_Y}
  && GY}
\xymatrix@R=3ex@C=4ex
    {
     FX \ar[rr]^{\eta_X}
        \ar@<-10pt>[dd]_{Fg}
        \ar@<10pt>[dd]^{ff}
        \ar@{}[dd]|{\underset{\Longleftarrow}{F\alpha}}
  && GX \ar[dd]^{Gf}
        \ar@{}[ddll]|{\rotatebox[origin=c]{45}{$\Leftarrow$}
        \eta_f}
\\ \\
     FY \ar[rr]_{\eta_Y}
  && GY}
$$
\end{enumerate}
A \emph{2-natural} transformation is a pseudonatural transformation such that $\eta_f$ is the identity for every $f$, in which case the first two conditions are trivial and the third is the axiom of 2-naturalness.
    
\vspace{1ex}

Suppose now that $\tau,\sigma: F \Longrightarrow G$ are pseudonatural transformations between 2-functors $F, G: \mathscr{C} \Longrightarrow \mathscr{D}$. A \emph{modification} \mbox{$\mu : \tau \longrightarrow \sigma$ assigns to each object $X$} of $\mathscr{C}$ a 2-cell $\mu_X : \tau_X \Longrightarrow \sigma_X$ in $\mathscr{D}$, such that for
any arrow $f: X \longrightarrow Y$
the equality $\sigma_f \circ (Gf \,\mu_X) = (\mu_Y Ff) \circ \tau_f$ is verified
 :
$$
\xymatrix@R=3ex@C=4ex
    {
     FX 
        \ar@<-7pt>[rr]_{\sigma_X}
        \ar@<7pt>[rr]^{\tau_X}
        \ar@{}[rr]|{\Downarrow \, \mu_X}
        \ar@{}[rrdd]|{\rotatebox[origin=c]{45}{$\Leftarrow$}
                                                 \sigma_{f}}
        \ar[dd]^{ff}
  && GX \ar[dd]^{Gf \hspace{4ex} = \hspace{3ex}}
\\ \\
     FY \ar[rr]_{\sigma_Y}
  && GY}
\xymatrix@R=3ex@C=4ex
    {
     FX \ar[dd]^{Ff}
        \ar[rr]^{\tau_X}
  && GX \ar[dd]^{Gf}
        \ar@{}[ddll]|{\rotatebox[origin=c]{45}{$\Leftarrow$}
        \tau_f}
\\ \\
     FY 
        \ar@<-7pt>[rr]_{\sigma_Y}
        \ar@<7pt>[rr]^{\tau_Y}
        \ar@{}[rr]|{\Downarrow \, \mu_Y}
  && GY}
$$
When $\tau,\,\sigma$ are 2-natural, the above equality reduces to $\mu_Y Ff = Gf \,\mu_X$
    
\end{definitions}

\begin{notation} \label{strict_&_pseudo}
2-natural (pseudonatural) transformations and modifications compose vertically and horizontally (see, eg, \cite{Bak} I.2.4, p. 25).
If $\mathscr{C}$ and $\mathscr{D}$ are 2-categories, we denote by$Hom_p(\mathscr{C}, \mathscr{D})$ the 2-category  where the objects are 2-functors from $\mathscr{C}$ into $\mathscr{D}$, the arrows are the pseudonatural transformations and the 2-cells are the modifications, and
we denote $Hom_s(\mathscr{C}, \mathscr{D})$ when we consider only the 2-natural transformations.
\end{notation}

\subsection{Equivalences}

\begin{definition}
We say that an arrow $f: X \mr{} Y$ in a 2-category is an \emph{equivalence} if there is an arrow $g: Y \mr{} X$ (which is called \emph{quasi -inverse}) and two invertible cells $\eta: id_X \Rightarrow gf$, $\varepsilon: fg \Rightarrow id_Y $.
\end{definition}
The quasi-inverse $g$ is determined up to an invertible 2-cell, and one can allways choose $g$, $\eta$,
    $\varepsilon$ such that the \emph{triangular equations} are verified
    $(\varepsilon\,f) \circ (f\,\eta) = id_f$
    $(g\,\varepsilon) \circ (\eta\,g) = id_g$.
    Although any quasi-inverse is also an equivalence, we consider important to consider that equivalences have a direction.

Recall the following non-trivial and fundamental fact:
\begin{proposition} \label{ffes=equi}
In the 2-category {\bf Cat} a functor
$F: \cc{X} \mr{F} \cc{Y}$ is an equivalence if and only if it is fully faithful and essentially surjective.
\end{proposition}

If $s: X \longrightarrow Y$ on a 2-category $\mathscr{C}$ is an equivalence, then it induces equivalences on the categories of maps:
For each object $Z$ in $\mathscr{C}$, the functor $s^*: \mathscr{C}[Y, Z] \longrightarrow \mathscr{C}[X, Z]$ given by the precomposition  $f \longmapsto fs$ is an equivalence of categories, with quasi-inverse $t^*$ given by any quasi-inverse $t$ of $s$.
In the same way it can be seen that the post-composition
\mbox{$s_*: \mathscr{C}[Z, X] \longrightarrow \mathscr{C}[Z, Y]$} is an equivalence of categories.
Furthermore, we have the following
\begin{proposition} \label{s*_equivalence}
    Given $s:X \longrightarrow Y$ in a 2-category $\mathscr{C}$, then $s$ is an equivalence if and only if $s^*: \mathscr{C}[Y , Z] \longrightarrow \mathscr{C}[X, Z]$ is an equivalence for all $Z$ in $\mathscr{C}$ if and only if $s_*: \mathscr{C}[Z, X] \longrightarrow \mathscr{C}[Z, Y]$ is an equivalence for all \mbox{$Z$ in $\mathscr{C}$.}
\end{proposition}

    \begin{definition} \label{pseudoequiv}
A 2-functor $F: \mathscr{C} \longrightarrow \mathscr{D}$ is a \emph{pseudoequivalence of \mbox{2-categories}} if there exist $G: \mathscr{D} \longrightarrow \mathscr{C }$ and pseudo-natural transformations $\eta: Id_{\mathscr{C}} \Longrightarrow GF$, $\theta: FG \Longrightarrow Id_{\mathscr{D}}$ which are equivalents in $Hom_p(\mathscr{C} , \mathscr{C})$ and in $Hom_p(\mathscr{D}, \mathscr{D})$, respectively.
The 2-functor $F$ would be a pseudoequivalence of 2-categories in the strict sense when $\eta$ and $\theta$ are equivalences in $Hom_s(\mathscr{C}, \mathscr{C })$ and in $Hom_s(\mathscr{D}, \mathscr{D})$, respectively (notation
\ref{strict_&_pseudo}).
    \end{definition}

The following fact is frequently used in the literature, for a detailed demonstration of it see \cite{GJ}.
\begin{proposition} \label{prop_infinite}
   Let $\eta: F \Longrightarrow G : \mathscr{C} \longrightarrow \mathscr{D}$ be a pseudonatural transformation between 2-functors. Then $\eta$ is an equivalence on $Hom_p(\mathscr{C}, \mathscr{D})$ if and only if each component $\eta_X$ is an equivalence on the 2-category 
   $\mathscr{D}$.
\end{proposition}

We note that the previous proposition does not hold for 2-natural transformations. That is, a 2-natural transformation that is a pointwise equivalence is not necessarily an equivalence on $Hom_s(\mathscr{C}, \mathscr{D})$.

\subsection{2-localizations}
Adapting to 2-categories the  definition in \cite{Pro} for bi-categories we have
\begin{definition} \label{2-location}
    Let $\mathscr{C}$ be a 2-category and $\Sigma$ be a subclass of morphisms. The \emph{2-localization} of $\mathscr{C}$ at $\Sigma$ is a 2-category $\mathscr{C}[\Sigma^{-1}]$ together with a 2-functor $q: \mathscr{C} \longrightarrow \mathscr{C}[\Sigma^{-1}]$ such that
    \begin{enumerate}
        \item $q(s)$ is an equivalence for all $s \in \Sigma$;
        \item for every 2-category $\mathscr{D}$, $q$ induces a pseudoequivalence of 2-categories given by the precomposition
            $$q^*: Hom_p(\mathscr{C}[\Sigma^{-1}], \mathscr{D}) \longrightarrow Hom_{p}(\mathscr{C}, \mathscr{D})_+ \,,$$
        where $Hom_{p}(\mathscr{C}, \mathscr{D})_+$ consists of the 2-functors that send the elements of $\Sigma$ into equivalences.
    \end{enumerate}
\end{definition}
   The 2-category $\mathscr{C}[\Sigma^{-1}]$ is characterised up to  pseudoequivalences.


\subsection{Model Categories} \label{MC}

The model categories were introduced by Quillen (\cite{Qui}), the   definition that we will use is stronger than the original definition and it is the one that Quillen also introduced with the name of \emph {closed model category}.

\begin{definition} \label{facelift}
    Let $\mathscr{X}$ be a category. We say that a map $f$ on $\mathscr{X}$ has the \emph{left-lifting property} with respect to a map $g$ if every problem of the form
$$
\xymatrix
       {
        \cdot \ar[r] \ar[d]_{f} & \cdot  \ar[d]^{g}
      \\
        \cdot \ar[r] \ar@{-->}[ur]^{h} & {\cdot}
        }
$$
has a solution $h$, not necessarily unique, which makes both triangles commute. Equivalently, we say that $g$ has the \emph{right-lifting property} with respect to $f$.
\end{definition}

\begin{definition} \label{retract}
    Given $f:X \longrightarrow Y$, $g: X' \longrightarrow Y'$ in a category $\mathscr{X}$, then \emph{f is a retract of g} if there is a commutative diagram:
    \begin{center}
        $\xymatrix{
        X \ar[r] \ar@/^1pc/[rr]^{id_X} \ar[d]_{f} & X' \ar[r] \ar[d]_{g} & X \ar [d]^{f} \\
        Y \ar[r] \ar@/_1pc/[rr]_{id_Y} & Y' \ar[r] & Y \\
        }$
    \end{center}
\end{definition}

The following is a definition considered in \cite{JG}, \cite{Dwy}, and introduced by Quillen (\cite{Qui}) in a different but equivalent way,

\begin{definition} \label{model_cat}
    A \emph{model category} is a category $\mathscr{C}$ provided with three classes of morphisms $\mathcal{F}$, $\mathcal{C}o\!f$ and $\mathcal {W}$, which we call, respectively, Fibrations, Cofibrations and Weak Equivalences, satisfying the following axioms.
    \renewcommand{\labelenumi}{M\arabic{enumi}.}
    \renewcommand{\labelenumii}{\roman{enumii}.}
    \begin{enumerate}
        \item $\mathscr{C}$ has finite limits and finite colimits. 
\footnote{Nowadays many authors require the existence of all \emph{small} limits and colimits.}
        \item If a cofibration is also a weak equivalence, then it has the left lifting property with respect to any fibration.
        
        If a fibration is also a weak equivalence, then it has the right-lifting property with respect to any cofibration.
        \item If $f$ is a retract of $g$ and $g$ is either a fibration, a cofibration, or a weak equivalence, then so is $f$. Furthermore, the three classes are closed by composition and contain the identities.
        \item Every map $f$ in $\mathscr{C}$ can be factored $s=pi$, where $p$ is a fibration and $i$ a cofibration, in two ways, one with $p$ a weak equivalence, the other with $i$ a weak equivalence.
        \item Sean \begin{tikzcd}
                        X \arrow[r, "f"] & Y \arrow[r, "g"] & Z
                    \end{tikzcd}
         in $\mathscr{C}$. If any two of the three maps $f$, $g$ and $gf$ are weak equivalences, then all three are (this axiom is often called the property of $"3 \; for \; 2"$).
    \end{enumerate}
\end{definition}
\begin{sinnada} {\bf Determination.}
Any two of the three distinguished classes of morphisms determine the third, that is, the model structure is determined by two of the three distinguished classes $\mathcal{F}$, $\mathcal{C}o\!f$, $\mathcal{ W}$.
\end{sinnada}
\begin{sinnada} \label{duality}
{\bf Duality.} The axioms of the \ref{model_cat} definition are self-dual. Given a model category $\mathscr{C}$, the opposite category $\cc{C}^{op}$ admits a model structure, where
$\cc{W}^{op} = \cc{W}$,
$\cc{F}^{op} = \mathcal{C}o\!f$,
$(\mathcal{C}o\!f)^{op} = \cc{F}$.
\end{sinnada}

We note that, by axiom M1, in a model category we always have an initial object and a terminal object, denoted $0$ and $1$, respectively.

\begin{definition}
   An object $X$ in a model category $\mathscr{C}$ is \emph{fibrant} if $X \longrightarrow 1$ is a fibration, and \emph{cofibrant} if $0 \longrightarrow X$ is a cofibration, a fibration or a cofibration are \emph{trivial} if they are also weak equivalences.
\end{definition}

\begin{notation}
    We will use the following notation:
    \renewcommand{\labelenumi}{\roman{enumi}.}
    \begin{enumerate}
  \item $\cdot \mre{} \cdot$ (weak equivalences)
  \item $\cdot \mrf{} \cdot$ (fibrations)
    \hspace{6ex} $\cdot \mrfe{} \cdot$ (trivial fibrations)
  \item $\cdot \mrc{} \cdot$ (cofibrations)
    \hspace{3ex} $\cdot \mrce{} \cdot$ (trivial cofibrations)
    \end{enumerate}
\end{notation}

Let us recall the following results, whose proofs can be seen, for example, in
\cite{JG}.


\begin{sinnada} \label{lifting_sii}
    Let $\mathscr{C}$ be a model category.
    \renewcommand{\labelenumi}{\roman{enumi}.}
    \begin{enumerate}
        \item A map in $\mathscr{C}$ is a cofibration (trivial cofibration) if and only if it has the left-lifting property with respect to all trivial fibration (fibration).
        \item A morphism in $\mathscr{C}$ is a fibration (trivial fibration) if and only if it has the right-lifting property with respect to every trivial cofibration (cofibration).
    \end{enumerate}
\end{sinnada}

\begin{sinnada}
A map $f$ is a weak equivalence if and only if it admits a factorization $f=pi$, where $p$ is a trivial fibration and $i$ is a trivial cofibration.
\end{sinnada}

\begin{sinnada} \label{stablebypullback}
    If $\mathscr{C}$ is a model category, then the classes of cofibrations and trivial cofibrations are both pushout stable. Fibrations and trivial fibrations are both pullback-stable.
\end{sinnada}

\begin{sinnada}
    While $\mathcal{F}$ ($\mathcal{F}\cap \mathcal{W}$) and $\mathcal{C}o\!f$ ($co\mathcal{F}\cap \mathcal{W}$ ) are classes closed by pullbacks and pushouts, respectively, it is \emph{not} true that the class $\mathcal{W}$ of weak equivalences has any of these properties.
\end{sinnada}

\begin{definition}
    A \emph{weak factorization system} on a category $\mathscr{X}$ is a pair $(\mathcal{L}, \mathcal{R})$ of distinguished classes of maps such that
    \begin{enumerate}
        \item Every map $h$ in $\mathscr{X}$ can be factored as $h=gf$, with $f \in \mathcal{L}$ and $g \in \mathcal{R}$.
        \item $\mathcal{L}$ is precisely the class of maps that have the left-lifting property with respect to every map of $\mathcal{R}$.
        
            $\mathcal{R}$ is the class of maps that have the right-rise property with respect to every map in $\mathcal{L}$.
    \end{enumerate}
\end{definition}

\emph{If $\mathscr{C}$ is a model category, the classes $(\mathcal{C} \cap \mathcal{W}, \mathcal{F})$ and $(\mathcal{C}, \mathcal{F} \cap \mathcal{W})$ are both a weak factorization system.}

\begin{definition} \label{functorial} ${}$

{\bf 1.} We say that a weak factorization system 
$(\mathcal{L}, \mathcal{R})$
is \emph{functorial} if for each $f$ there is a chosen factorization
$(\lambda_f, \rho_f)$,
$\lambda_f \in \mathcal{L}$,
$\rho_f \in \mathcal{R}$,
$f = \rho_f \circ \lambda_f$,
in such a way that every time we have a commutative diagram like here below on the left, we have a map
$F(u, v)$ making commutative the \mbox{diagram on the right,}
$$
\xymatrix@C=8ex@R=3ex
       {
        \cdot \ar[r]^{u} \ar[d]_{f}
      & \cdot \ar[d]^{g}
        \\
        \cdot \ar[r]_{v}
     & \cdot
       }
\hspace{12ex}
\xymatrix@C=8ex@R=3ex
       {
        \ar[r]^{u} \cdot \ar[d]_{\lambda_f} \ar@/_1.8pc/[dd]_{f}
     & \cdot \ar@/^1.8pc/[dd]^{g} \ar[d]^{\lambda_g}
     \\
        \cdot \ar[r]_{F(u, v)} \ar[d]_{\rho_f}
     & \cdot \ar[d]^{\rho_g}
     \\
        \cdot \ar[r]_{v}
     & \cdot
       }
$$
and such that $F(u, v)$ depends functorialy on $u$ and $v$, that is, if $f=g$ then \mbox{$F(id, id)=id$,} and 
\mbox{$F(u\circ u', v\circ v') = F(u, v) \circ F(u', v')$}.

\vspace{1ex}

{\bf 2.} We say that the factorization is \emph{normal} if for
$f \in \cc{L}$, $\lambda_f = id$, and for
$f \in \cc{R}$, $\rho_f = id$.

\vspace{1ex}

{\bf 3. Exercise.} \emph{The functorial factorization delivered by the small object argument in cofibrantly generated categories is normal. Furthermore, given any $f$, it is \emph{initial} in the poset of the factorizations of $f$.} We leave it to the reader to verify this by reviewing any of the proofs of the small object argument in the literature.
\end{definition}

\begin{remark}
    We denote by $\overrightarrow{\mathscr{C}}$ the category whose objects are the morphisms of $\mathscr{C}$ and an arrow from $f$ to $g$ in $\overrightarrow{\mathscr{C}} $ is a pair $(u, v)$ of maps of $\mathscr{C}$ such that $gu=vf$. Let $dom, codom: \overrightarrow{\mathscr{C}} \longrightarrow \mathscr{C}$ be the functors that choose the domain arrow and the codomain arrow respectively. The above definition tells us precisely that a system $(\mathcal{L}, \mathcal{R})$ is functorial if there exists a functor
$F:\overrightarrow{\mathscr{C}} \longrightarrow \mathscr{C}$ and natural transformations $\lambda: dom \longrightarrow F$ and $\rho: F \longrightarrow codom$ such that for all $f$ in
 $\mathscr{C}$ we have
$$
 \xymatrix@C=3ex@R=3ex
       {
        dom(f) \ar[rr]^{f} \ar[dr]_{\lambda_f}
     && codom(f)
      \\
      & F(f) \ar[ur]_{\rho_f}
       }
 \hspace{4ex} with \;\;
 \lambda_f \in \mathcal{L}, \;\; \rho_f \in \mathcal{R}.
$$
    We say that $(F, \lambda, \rho)$ is a \emph{functorial realization} for the weak factorization $(\mathcal{L}, \mathcal{R})$. See for example \cite{Riehl}.
\end{remark}

\section{The homotopy 2-category \texorpdfstring{$\HAo$}{}}
\label{hacheo}

In this section, we adapt to the context of 2-localization of 1-categories using  2-functors 
the results on bilocalizations of bicategories using pseudofunctors developed in \cite{e.d.2}. 

We construct the 2-localization of a category $\sr{A}$  at  a class of morphisms $\Sigma$ containing the identities. The novel aspect over \cite{Pro} is that we do it using a notion of cylinders and homotopies, so we will call the members of $\Sigma$ \emph{weak equivalences}, and the 2-localization the homotopy 2-category. We use a definition of cylinder more general and less rigid than Quillen's definition, and consequently we will obtain a more general definition of homotopy.

Let us recall that the problem we want to solve consists of constructing a 2-category $\HAo$ together with a 2-functor $i: \sr{A} \longrightarrow \HAo$ which has the universal property of a 2-localization by Definition \ref{2-location}.

\subsection{Construction of \texorpdfstring{$\HAo$}{}}
\begin{definition} \label{cylindrical_homotopias}
      A \emph{cylinder} $C=(W, Z, d_0, d_1, s, x)$ for an object $X$ in $\mathscr{A}$ is a configuration
        $$\xymatrix{
     X \ar@<3pt>[rr]^{d_0}
        \ar@<-3pt>[rr]_{d_1}
        \ar[dr]_{x} && W,
        \ar[dl]^{s}| \circ\\
        & Z \\ }$$
    where $s \in \Sigma$ and $sd_0=sd_1=x$.

\vspace{1ex}
    
    A \emph{(left) homotopy} $H=(C, h)$ from $f$ to $g$ with cylinder  \mbox{$C=(W, Z, d_0, d_1, s, x)$} is an arrow $h:W \longrightarrow Y$ satisfying $hd_0=f$ and $hd_1=g$. We denote
    $H:\xymatrix{f \ar@2{~>}[r] & g}$.
\begin{equation} \label{diagramhomotopia}
    \xymatrix{
     X \ar@/^2pc/@<6pt>[rrrr]^{f}
        \ar@/^2pc/[rrrr]_{g}
        \ar@<2pt>[rr]^{d_0}
        \ar@<-3pt>[rr]_{d_1}
        \ar[dr]_{x} && W
        \ar[rr]^{h} \ar[dl]^{s}|\circ && Y. \\
        & Z \\ }
\end{equation}
\end{definition}
Throughout this article (except in the appendix) we work only with left homotopies, and thus omit to write the word "left".

{
We would like to take the homotopies as the 2-cells of a 2-category with the same objects and arrows as $\cc{A}$, but it turns out that the 2-category axioms do not hold. It is necessary, therefore, to define an adequate equivalence relation between homotopies, so that if we take the equivalence classes as \mbox{2-cells,} then the 2-category axioms are verified. It turns out that the \mbox{2-category} thus obtained has invertible 2-cells, i.e. it is a $(2,\,1)$-category in contemporary terminology (see lemma \ref{2_cells_inv}).
}

\subsubsection{Equivalence classes of homotopies.}
The equivalence relation that we are looking for comes from the following observation:

\begin{remark}
Let $\mathscr{D}$ be a 2-category and consider in $\mathscr{D}$ the diagram \ref{diagramhomotopia} now with
$s$ a true equivalence.
{In this case} $s$ induces a fully faithful functor, 
\begin{center}
    $\xymatrix{ \mathscr{D}[X, W] \ar[rr]^{s_{*}} && \mathscr{D}[X, Z]
    }$
\end{center}
(see \ref{ffes=equi}, \ref{s*_equivalence}),
so that we get a unique 2-cell $\widehat{C}: d_0 \Longrightarrow d_1$ such that $s\widehat{C}=x$.

Let $\xymatrix{ \mathscr{A} \ar[r]^{F} & \mathscr{D}}$ be a 2-functor that sends the arrows of the class $\Sigma$ into equivalences, and let $H= (C, h)$ be a homotopy from $f$ to $g$ on $\sr{A}$ with cylinder $C=(W, Z, d_0, d_1, s, x)$.
Applying $F$ to the diagram \ref{diagramhomotopia} above yields
$$
\xymatrix@R=3ex@C=3ex
        {
     FX \ar@/^2pc/@<6pt>[rrrr]^{Ff}
        \ar@/^2pc/[rrrr]_{Fg}
        \ar@<2pt>[rr]^{Fd_0}
        \ar@<-3pt>[rr]_{Fd_1}
        \ar[dr]_{Fx} && FW
        \ar[rr]^{Fh} \ar[dl]^{Fs} && FY, \\
        & FZ \\}
$$
which is a commutative diagram in $\mathscr{D}$.
Since $Fs$ is an equivalence, there is a unique 2-cell $\widehat{FC}: Fd_0 \Longrightarrow Fd_1$ such that $Fs\widehat{FC}=Fx$.

\end{remark}

\begin{definition} \label{FH}
    Given a homotopy $H=(C, h)$ on $\sr{A}$ and given a 2-functor $F: \sr{A} \longrightarrow \mathscr{D}$ which sends the class $ \Sigma$ into equivalences, we define a 2-cell $\widehat{FH}$ in $\mathscr{D}$ as $\widehat{FH}:=Fh\widehat{FC} : Ff \Longrightarrow Fg$, where $ \widehat{FC}: Fd_0 \Longrightarrow Fd_1$ is the unique 2-cell that satisfies $Fs\widehat{FC}=Fx$.
\end{definition}


\begin{definition} \label{rel_adhoc}
    Let $f, g: X \longrightarrow Y$ morphisms in $\sr{A}$ and $H, H': \xymatrix{f \ar@2{~>}[r] & g}$ be two homotopies. We say that $H \sim H'$ if and only if $\widehat{FH}=\widehat{FH'}$ for every 2-functor $F: \sr{A} \longrightarrow \mathscr{D}$ such that $F(\Sigma) \subseteq Equiv(\mathscr{D})$, for every 2-category $\mathscr{D}$.

\vspace{1ex}
    
    \emph{{Note that this definition of equivalence is not intrinsic to $\sr{A}$, it is \emph{semantic} in the sense that it quantifies over 2-functors in the category of 2-categories.}}
\end{definition}

It is clear that we can establish an equivalence relation between \mbox{sequences} of composable homotopies. We define a 2-cell in $\HAo$ as the class $[H_n,..., H_1]$ of a finite sequence of homotopies $\xymatrix{ f_0 \ar@2{~>}[r]^ (.3){H_1} & f_{1} \ldots f_{n-1} \ar@2{~>}[r]^(.7){H_n} & f_n }$, where:

$(H_n, ..., H_1) \sim (K_m, ..., K_1)$ if and only if 
for every $2$-functor $F:\sr{A} \mr{} \sr{D}$, 
$\widehat{FH_n} \circ ... \circ \widehat{FH_1} = \widehat{ FK_m}\circ ... \circ \widehat{FK_1}$.

\vspace{1ex}

Let us see that $\HAo$ is thus effectively a 2-category.
\subsubsection{Vertical Composition}
$\HAo(X,Y)$ is a category for each pair of objects $X$, $Y$ in $\sr{A}$:

We define the vertical composition of 2-cells as the juxtaposition of the sequences.
\begin{center}
    $[H_n, ..., H_1] \circ [K_m, ..., K_1] =
 [H_n, ..., H_1,\,K_m, ..., K_1]$.
\end{center}

The associativity is an immediate consequence of the associativity of the vertical composition in $\mathscr{D}$.

   Note that  the single-homotopy classes generate the 2-cells in $\HAo$.

\begin{sinnada} \label{hpy_identity}
{\bf Identities for the vertical composition.}
Let $f \in \sr{A}[X,Y]$ and let $H: \xymatrix{f \ar@2{~>}[r] & f}$ be a homotopy with cylinder $C$ .
By how the equivalence relation between homotopy sequences is defined, it is clear that
    $[H]$ will be the identity of $f$ in $\HAo(X,Y)$ if and only if
    $\widehat{FH}=Ff$ in $\mathscr{D}(FX,FY)$ for all $F \in Hom_+(\sr{A}, \mathscr{D}).$

Thus, for example, either of the following two homotopies determines the vertical identity for $f$:
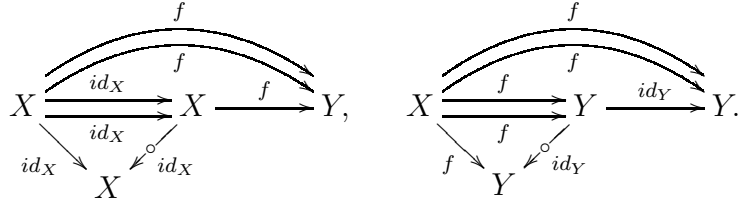
$$
\xymatrix@R=3ex@C=3ex{
     X \ar@/^2pc/@<6pt>[rrrr]^{f}
        \ar@/^2pc/[rrrr]_{f}
        \ar@<3pt>[rr]^{id_X}
        \ar@<-3pt>[rr]_{id_X}
        \ar[dr]_{id_X} && X
        \ar[rr]^{f} \ar[dl]^{id_X}|\circ && Y, \\
        & X \\ }
\hspace{3ex}
\xymatrix@R=3ex@C=3ex{
     X \ar@/^2pc/@<6pt>[rrrr]^{f}
        \ar@/^2pc/[rrrr]_{f}
        \ar@<3pt>[rr]^{f}
        \ar@<-3pt>[rr]_{f}
        \ar[dr]_{f} && Y
        \ar[rr]^{id_Y} \ar[dl]^{id_Y}|\circ && Y. \\
        &  Y \\  }
$$
\end{sinnada}   
                                             
\subsubsection{Horizontal composition} \label{horiz_composition_def}
Since we already have a vertical composition, if we define the horizontal compositions $l\,[H]=[I_l]\,[H]$ and $[H]\,r=[H]\,[I_r]$, the horizontal composition of general 2-cells is obtained as follows, see
 Remark \ref{horizontal_comp}:
 $$[H']\,[H] \underset{\text{def}}{=}
 [H']\,g \circ f'\,[H] \underset{(1)}{=}
 g'\,[H] \circ [H']\,f$$
 as long as the equality $(1)$ is verified.

Let 
$
\xymatrix
     {
      X' \ar[r]^{l} 
    & X \ar@<4pt>[r]^{f} \ar@<-3pt>[r]_{g} 
    & Y \ar[r]^{r} & Y' 
      }
$
y $H= (C, h):\xymatrix{f \ar@2{~>}[r] & g}$
a homotopy with cylinder $C=(W, Z, d_0, d_1, s, x)$.
We consider 
$Hl=(Cl, h):$

$ 
\xymatrix
    {
     fl \ar@2{~>}[r] & gl
    }
$
and $rH=(C, rh):$
  
$ 
\xymatrix
     {
      rf \ar@2{~>}[r] 
    & rg
      }   
$
where $Cl=(W, Z, d_0l, d_1l, s, xl)$ is a cylinder for $X'$.
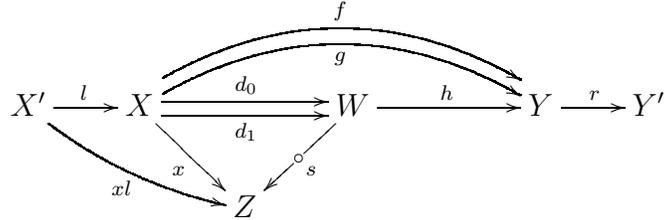
\begin{center}
    $\xymatrix{
     X'\ar[r]^{l} \ar@/_0.5pc/[drr]_{xl}
     & X \ar@/^2pc/@<6pt>[rrrr]^{f}
        \ar@/^2pc/[rrrr]_{g}
        \ar@<2pt>[rr]^{d_0}
        \ar@<-3pt>[rr]_{d_1}
        \ar[dr]_{x} && W
        \ar[rr]^{h} \ar[dl]^{s}|\circ && Y \ar[r]^{r}& Y'\\
        && Z \\ }
     $
\end{center}

It is clear that both $Hl$ and $rH$ are homotopies. In addition, we have the equations
\begin{equation} \label{eq}
    \widehat{F(Hl)}=\widehat{FH}Fl , \hspace{3ex} \widehat{F(rH)}=Fr\widehat{FH},
\end{equation}
since $\widehat{F(Hl)}=Fh\widehat{F(Cl)}$ and $Fs(\widehat{FC}Fl)=(Fs\widehat{FC})Fl=xl$, by the uniqueness of $\widehat{F(Cl)}$ in the definition \ref{FH} we have that $\widehat{F(Cl)}=\widehat{FC}Fl$, and then $\ \widehat{F(Hl)}=\widehat{FH}Fl$. The second equation is obvious, since $H$ and $rH$ have the same cylinder. We define:
\begin{equation} \label{dfeq}
[H]\,l=[Hl], \hspace{3ex} r\,[H]=[rH].
\end{equation}
Now, if $H \sim H'$, then
\begin{center}
$\widehat{F(Hl)}=\widehat{FHFl}=\widehat{FH}Fl=\widehat{FH'}Fl=\widehat{FH'Fl}=\widehat{F(H'l)}$ .
\end{center}
This tells us that the composition with $l$ is well defined, and in the same way we see the well definition of the composition with $r$.

More generally, for any sequence $[H_n, ..., H_1]$ of composable homotopies, we define:
$$
[H_n, ..., H_1]\,l = [H_nl, ..., H_1l], \;\;\;\;
r\,[H_n, ..., H_1] = [rH_n, ..., rH_1]
$$
From the equations \ref{eq}, and from the corresponding valid equation in $\sr{D}$, it follows the required equality $(1)$, that is, the third axiom in the Remark \ref{horizontal_comp}. On the other hand, we have by definition:
\begin{center}
    $([K]\,l)\circ ([H]\,l)=([Kl])\circ ([Hl])=[Kl,Hl]=[K, H]\,l=([ K]\circ [H])\,l$,
\end{center}

\begin{center}
    $[I_f]\,l=[I_fl]=[I_{fl}].$
\end{center}
Thus all the axioms of \ref{horizontal_comp}
are checked, so the horizontal composition of 2-cells in $\HAo$ is determined and it is compatible with the vertical composition.

By virtue of the same equations (\ref{eq}) it follows that the horizontal composition is also associative, and the identities in this case are,
for each object $X$, the 2-cell $[I_{id_X}]$, which we write $[I_X]$ to simplify the notation, $\widehat{FI_X} = id_X$.

\subsection{The fundamental property of \texorpdfstring{$\HAo$}{}} 
\label{fundamental}
 
The 2-category $\HAo$ thus obtained comes equipped with a 
\mbox{2-functor} \mbox{$i$: \mbox{$\sr{A} \mr{} \HAo $}} given by the inclusion, and although in general it will not send weak equivalences into equivalences, it has the following \mbox{\emph{partial universal property}.}

\begin{proposition} \label{pu_i}
    Let $i: \sr{A} \longrightarrow \HAo $ be the inclusion, $\mathscr{D}$ be a 2-category, and $F: \sr{A} \longrightarrow \mathscr{D}$ be a 2-functor that sends the elements of $\Sigma$ into equivalences. Then, there exists a unique 2-functor $\widetilde{F}: \HAo \longrightarrow \mathscr{D}$ such that $\widetilde{F}X=FX$ and $\widetilde{F}f=Ff$:
$$\xymatrix{
        \ar @{} [drr] |(.45){}
        \sr{A} \hspace{1mm} \ar@{^{(}->}[rr]^{i} \ar[dr]_{F} && \HAo \ar@{-->}[dl]^(.5){\exists!\widetilde{F}} \\
        & \mathscr{D} & \\
    }
$$
\end{proposition}

\begin{proof}
    We define $\widetilde{F}$ in the 2-cells as follows:
\begin{equation} \label{functor_hpy}
\widetilde{F}([H_n, ..., H_1]) = \widehat{FH_n} \circ ... \circ \widehat{FH_1},
\hspace{1ex} \text{in particular} \hspace{1.5ex}
\widetilde{F}([H])= \widehat{FH}.
\end{equation}
In this way, by the very definition of the equivalence relation $\widetilde{F}$ is well defined and is functorial for the vertical composition.

\vspace{1ex}
    
    Since $\widetilde{F}l\widetilde{F}([H])=Fl\widehat{FH}=\widehat{FlFH}=\widehat{F(lH)}= \widetilde{F}([lH] )$ and, analogously, $\widetilde{F}([H])\widetilde{F}r=\widetilde{F}([Hr])$, from the definition in \ref{horiz_composition_def} it follows that $\widetilde{F}$ is functorial with respect to horizontal composition.
    
\vspace{1ex}
    
    We now want to see the uniqueness of $\widetilde{F}$.
    
    Let $R:\HAo \longrightarrow \mathscr{D}$ be a 2-functor such that $Ri=F$.
    If $H=(C, h)$ is a homotopy on $\sr{A}$ with cylinder $C=(W, Z, d_0, d_1, s, x)$, we write $H=hH_0$ where $H_0 =(C, id_W)$.
    \begin{center}
       $\xymatrix{
        X
        \ar@<3pt>[rr]^{d_0}
        \ar@<-3pt>[rr]_{d_1}
        \ar[dr]_{x} && W
        \ar[r]^{id_W} \ar[dl]^{s}|\circ & W \ar[r]^{h} & Y\\
        & Z\\ }$
    \end{center}
    
    Since $R([H])=RhR([H_0])=FhR([H_0])$ and $\widetilde{F}(H)=\widehat{FH}=Fh\widehat{FC}$, to check that $R$ equals $\widetilde{F}$ in $[H]$, it suffices to see that $R([H_0])=\widehat{FC}$.
    
    We know that $\widehat{FC}: Fd_0 \Longrightarrow Fd_1$ is unique such that $Fs\widehat{FC}=Fx$.
    On the other hand 
    $FsR([H_0])=RsR([H_0])=R(s[H_0])=R([sH_0])=RI_x=Rx=Fx$ (since $[sH_0]=I_x$). It follows $R([H_0])=\widehat{FC}$.
\end{proof}

We will denote $Hom_{p}(\sr{A},\mathscr{D})_+$ and $Hom_{s}(\sr{A},\mathscr{D})_+$ to the subcategories of $Hom_p (\sr{A}, \mathscr{D})$ and $Hom_s(\sr{A}, \mathscr{D})$, respectively, whose objects are the functors $F$ such that $F(\Sigma) \subseteq Equiv(\mathscr{D})$.

The last proposition tells us that the precomposition with $i$
\begin{equation} \label{pre_hom_iso}
    i^*: Hom_{s}(\HAo, \mathscr{D})_+\longrightarrow Hom_{s}(\sr{A},\mathscr{D})_+
\end{equation}
determines a bijection between the objects, in particular, it is essentially surjective. Let us see that it is also fully faithful.

\begin{lemma} \label{key}
    Let $\mathscr{D}$ be a 2-category, $F, G \in Hom(\sr{A},\mathscr{D})_+$ and a cylinder $C=(W, Z, d_0, d_1 , s, x)$ for an object $X \in \sr{A}$. If $\theta: F \Longrightarrow G$ is a natural transformation, then $\theta_W\widehat{FC}=\widehat{GC}\theta_X$ holds.
\end{lemma}
$$
\xymatrix
   {
    FX \ar[rr]_{\theta_X}
       \ar@<-10pt>[dd]_{Fd_0}
       \ar@{}[dd]|{\overset{\widehat{FC}}{\Rightarrow}}
       \ar@<10pt>[dd]^{Fd_1}
 && GX \ar@<-10pt>[dd]_{Gd_0}
       \ar@{}[dd]|{\overset{\widehat{GC}}{\Rightarrow}}
       \ar@<10pt>[dd]^{Gd_1}
 \\\\
    FW \ar[rr]_{\theta_W}
 && GW }
$$

\begin{proof}
    Since $\widehat{GC}:Gd_0 \Longrightarrow Gd_1$ is the \'only 2-cell that satisfies $Gs\widehat{GC}=Gx$, then $\widehat{GC}\theta_X$ is the only one such that $ Gs(\widehat{GC}\theta_X)=Gx\theta_X$. Then, it is enough to see that $Gs(\theta_W\widehat{FC})=Gx\theta_X$.
    
    From the naturality of $\theta$ we obtain $Gs\theta_W=\theta_ZFs$ and $\theta_ZFx=Gx\theta_X$.
Therefore,
        $Gs(\theta_W\widehat{FC})=(Gs\theta_W)\widehat{FC}=(\theta_ZFs)\widehat{FC}=\theta_ZFx=Gx\theta_X$.
\end{proof}

\begin{proposition} \label{prop_key}
    If $F, G: \sr{A}\longrightarrow\mathscr{D}$ send the arrows of $\Sigma$ into equivalences and $\theta: F \Longrightarrow G$ is a natural transformation, then there exists a unique 2-natural transformation  $\widetilde{\theta}:\widetilde{F}\Longrightarrow\widetilde{G}$ such that $\widetilde{\theta}i=\theta$.
\end{proposition}

\begin{proof}
    For each $X$ in $\sr{A}$, we know that $\widetilde{F}i=F$ and $\widetilde{G}i=G$. We define $\widetilde{\theta}_X :\widetilde{F}X \longrightarrow \widetilde{G}X$ as $\widetilde{\theta}_X=\theta_X$, and let us see that $\widetilde{\theta }$ satisfies the 2-naturalness conditions.
    
    Let $f, g: X \longrightarrow Y$ and $[H]: f \Longrightarrow g$ be a 2-cell in $\HAo$. We want to prove the equality $\widetilde{\theta}_Y\widetilde{F}[H]=\widetilde{G}[H]\widetilde{\theta}_X$. If $H=(C, h)$, from the naturalness of $\theta$ we have $\theta_YFh=Gh\theta_X$, so $\theta_Y\widehat{FH}=\theta_YFh\widehat{FC}= Gh\theta_W\widehat{FC}$. Also, $\widehat{GH}\theta_X=Gh\widehat{GC}\theta_X$. From the previous lemma we get $\theta_W\widehat{FC}=\widehat{GC}\theta_X$ and then $\theta_Y\widehat{FH}=\widehat{GH}\theta_X$, which by definition of $\widetilde {F}, \widetilde{G}$ and $\widetilde{\theta}$, is exactly what we wanted to see.
\end{proof}

As a consequence of the previous results, we obtain that the precomposition with $i$, diagram \ref{pre_hom_iso}, induces, for every 2-category $\mathscr{D}$, an equivalence of categories, which is in fact an isomorphism.

Regarding the 2-categorical aspect, we have the following

\begin{lemma} \label{categorical_2aspect}
     Let $F, G  \in Hom_+(\sr{A}, \mathscr{D})$. If $\eta$, $\theta$: $F \Longrightarrow G$ are natural transformations and $\mu: \eta \longrightarrow \theta$ a modification, then there is a unique $\Tilde{\mu}: \Tilde{\eta} \longrightarrow \Tilde{\theta}$ such that $\Tilde{\mu}i=\mu$.
\end{lemma}

\begin{proof}
    We define $\Tilde{\mu}: \Tilde{\eta} \longrightarrow \Tilde{\theta}$ as $\Tilde{\mu}_X=\mu_X$ for every $X$.
    Let \mbox{$f, g: X \longrightarrow Y$}, and let $H$ be a homotopy from $f$ to $g$. We want to see that $\Tilde{\mu}_Y\Tilde{F}[H]=\Tilde{G}[H]\Tilde{\mu}_X$; that is, $\mu_Y\widehat{FH}=\widehat{GH}\mu_X$. Since $\mu$ is a modification, $\theta_YFg=Gg\mu_X$, and by the lemma \ref{key} we also have $\eta_W\widehat{FC}=\widehat{GC}\eta_X$, then :

\vspace{1ex}

$\mu_Y\widehat{FH} =\mu_YFg \circ \eta_Y \widehat{FH}
                          =Gg\mu_X\circ\eta_YFh\widehat{FC}
                          =Gg\mu_X \circ Gh\eta_W\widehat{FC}
                          = Gg\mu_X \circ Gh\widehat{GC}\eta_X
                          =Gg \mu_X \circ \widehat{GH}\eta_X
                          =\widehat{GH}\mu_X.
$
\end{proof}

The arguments in the proofs of \ref{pu_i}, \ref{key}, \ref{prop_key}, and \ref{categorical_2aspect}  follow mutatis mutandis in terms of
2-functors and pseudonatural transformations. We then have the main result of this section:
\begin{theorem} \label{hom_iso}
The inclusion 2-functor $i: \sr{A} \mr{} \HAo$ induces  2-category isomorphisms
\mbox{$i^*: Hom_{s}(\HAo, \mathscr{D})_+\longrightarrow Hom_{s}(\mathscr{A}, \mathscr{D})_+$,} and
\mbox{$i^*: Hom_{p}(\HAo, \mathscr{D})_+\longrightarrow Hom_{p}(\mathscr{A}, \mathscr{D})_+$.}\cqd
\end{theorem}

\subsection{A condition to have the 2-localization}

We see that from theorem \ref{hom_iso}
it follows that as soon as the morphisms of $\Sigma$ are equivalences in $\HAo$, the 2-functor
$\sr{A} \mr{i} \HAo$ will be the \mbox{2-localization} of $\sr{A}$ at $\Sigma$.
We now turn to consider a condition that assures us that this will be the case. 

\begin{lemma} \label{2_cells_inv}
Every 2-cell in $\HAo$ is invertible.
\end{lemma}

\begin{proof}
    Let $H=(C, h)$ be a homotopy from $f$ to $g$ with cylinder 
    \mbox{$C=(W, Z, d_0, d_1, s, x)$.} We define $H^{-1}=(C^{-1}, h)$, where $C^{-1}$ is the cylinder obtained from $C$ by exchanging $d_0$ and $d_1$, so that $H^{-1}$ is a homotopy from $g$ to $f$. Let's see that $[H^{-1}]\circ[H]=[I_f]$.
    
    Let $F: \sr{A}\longrightarrow\mathscr{D}$ be a 2-functor. We have \mbox{$\widehat{FC}: Fd_0 \Longrightarrow Fd_1$} and 
    \mbox{$\widehat{FC^{-1}}: Fd_1 \Longrightarrow Fd_0$} satisfying this equations \mbox{$Fs\widehat{FC}=Fx$} and \mbox{$Fs\widehat{ FC^{-1}}=Fx$.} Then  
\begin{center}$Fs(\widehat{FC^{-1}}\circ\widehat{FC}) = Fs\widehat{FC^{-1}}\circ
   Fs\widehat{FC}) =  Fx\circ Fx=Fx$.
\end{center}

    Since there is only one 2-cell from $Fd_0$ to $Fd_0$ satisfying this equality, it follows  $\widehat{FC^{-1}}\circ\widehat{FC}=id_{Fd_0}=Fd_0$. Thus,
    \begin{center}
    $\widehat{FH^{-1}}\circ \widehat{FH}=Fh\widehat{FC^{-1}} \circ Fh\widehat{FC}= Fh(\widehat{FC^{-1} } \circ \widehat{FC})=FhFd_0=Ff$;
    \end{center}
    that is, $[H^{-1}, H]=[I_f]$. A similar account shows that $[H, H^{-1}]=[I_g]$, therefore $[H]$ is invertible and its inverse is $[H]^{-1}=[H^{ -1}]$.
    \end{proof}

\begin{definition}
    A map $f:X \longrightarrow Y$ is a \emph{section} if it admits a left inverse; that is, there exists $g: Y\longrightarrow X$ that satisfies $gf=id_X$.
Dually, we say that $f$ is a \emph{retraction} if it has a right inverse. We say that a map is \emph{split} if it is a section or a retraction.
\end{definition}

\begin{definition} \label{2x1}
We say that a class $\Sigma$ of arrows satisfies the property \mbox{"2 for 1" \,} if given
$X \mr{f} Y \mr{g} X$ such that $g\,f = id_X$, then
\mbox{$f \in \Sigma$ $\iff$ $g \in \Sigma$.}
\end{definition}
We note that the well-known "3 for 2" property (axiom M5 in Definition \ref{model_cat}) implies the "2 for 1" property.
\begin{proposition} \label{split==>loc}
If the class $\Sigma$ satisfies the "2 for 1"\, property, then weak equivalences that are sections or retractions are equivalences in
$\HAo$.
\end{proposition}
\begin{proof}

Let $X \mr{s} Y$ be a weak equivalence that is a section, with $Y \mr{r} X$ a left inverse. To see that it is an equivalence in $\HAo$ we have to show that there is an isomorphism between $id_Y$ and the composition $s\,r$, but since in $\HAo$ every 2-cell is invertible, we only need to prove the existence of such a 2-cell, that is a homotopy $\xymatrix{ s\,r \ar@2{~>}[r] & id_Y}$.

Since $rs\,r=r$, the diagram
$$
\xymatrix
     {
      Y \ar@<3pt>[rr]^{sr}\ar@<-3pt>[rr]_{id_Y}\ar[dr]_{r} 
   && Y\ar[rr]^{id_Y}\ar[dl]^{r}|\circ 
   && Y
   \\
    & X \\
      }
$$
    is commutative and gives us a homotopy from $s\,r$ to $id_Y$. Note that the \mbox{2 for 1} hypothesis on $\Sigma$ tells us that $r$ is a weak equivalence, necessary fact to have a cylinder.
Since the case where $s$ is a retraction is completely analogous, we can conclude the proof.
\end{proof}

As the composition of equivalences is an equivalence, the Theorem \ref{hom_iso} together with the Proposition \ref{split==>loc} give us  (remember definition \ref{2-location}):

\begin{theorem} \label{2-loc-W-split}
 If the class $\Sigma$ satisfies the "2 for 1"\, property and every weak equivalence is a composition of split weak equivalences, the inclusion
 \mbox{$i:\sr{A}\longrightarrow \HAo$} is the 2-localization of $\sr{A}$ at the \mbox{class $\Sigma$.}
Furthermore, the 2-functors \mbox{$i^*: Hom_s(\HAo, \mathscr{D})\longrightarrow Hom_{s}(\sr{A}, \mathscr{D})_+$} and
    \mbox{$i^*: Hom_p(\HAo, \mathscr{D})\longrightarrow Hom_{p}(\sr{A}, \mathscr{D})_+$} are both isomorphisms of
    \mbox{2-categories.}
\end{theorem}
Homotopies (as in  \ref{cylindrical_homotopias})  were defined in \cite[3.2.3, 3.2.4]{S} and  
used to define a congruence between arrows rather than the 2-cells of a 2-category, and a quotient category {Ho}$(\cc{A})$ is obtained.  In view of  
\cite[3.2.6]{S} it follows that the connected components of the hom-categories of $\HAo$ are the arrows of 
{Ho}$(\cc{A})$, so $\HAo$ is a direct 2-dimensional generalization of Ho$(\cc{A})$.

\subsection{The germs of homotopies}

 Fixing two objects and two arrows 
$X \mrdos{f}{g} Y$ we consider for each cylinder $C$ for $X$ the set  $Hpy(X,Y)(f,g)(C)$ of all homotopies $f \Hpy{} g$ with cylinder $C$.
This actually determine a contravariant set valued functor defined on the category of cylinders for $X$, see \ref{cylmap} below. We call \emph{germ relation} the equivalent relation which computes the colimit set of this functor. A key property is that it is finer than the ad-hoc relation 
used to define $\HAo$.
\vspace{1ex}

We start by introducing the concept of \emph{cylinder map}, and we will prove that if two homotopies are connected by a zig-zag of maps between their cylinders, then they determine the same class, a fact that we will use many times to show the equality of 2-cells in $\HAo$.
\begin{definition}[{\bf Cylinder maps}] \label{cylmap}
    Let $C=(W, Z, d_0, d_1, s, x)$, 
 \mbox{$C'=(W', Z', d'_0, d'_1, s', x')$} be two cylinders for $X $ in 
 $\sr{A}$. A \emph{cylinder map} $C\longrightarrow C'$ consists of a pair of maps $\phi: W\longrightarrow W'$ and $\psi: Z \longrightarrow Z'$ that make the following diagram commutative
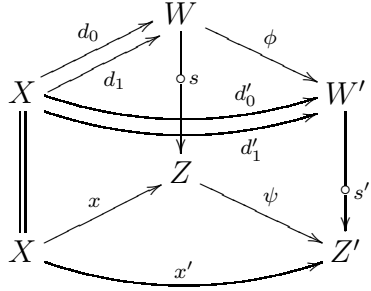
\begin{center}
    $\xymatrix@R=3ex@C=3.5ex
        {&& W \ar[dd]^(.4){s}|(.4)\circ
           \ar[drr]^{\phi} \\
        X \ar@<3pt>[urr]^{d_0}
          \ar@<-3pt>[urr]_{d_1}
          \ar@/_1pc/@<3pt>[rrrr]^(.7){d'_0}
          \ar@/_1pc/@<-3pt>[rrrr]_(.7){d'_1}
          \ar@{=}[dd] &&&& W'\ar[dd]^(.6){s'}|(.6)\circ \\
        && Z\ar[drr]^{\psi}\\
        X\ar[urr]^{x} \ar@/_1pc/[rrrr]^{x'} &&&& Z'}$
\end{center}
It is straightforward to check that they compose and form a category, and that $Hpy(X,Y)(f,g)(C)$ is a contravariant functor on the variable $C$.
\cqd
\end{definition}

\begin{definition} \label{rel_germ}
Given $H=(C, h)$ and $H'=(C', h')$ two homotopies from $f$ to $g$, we say that $H' \germ > H$  if there exists a morphism of cylinders
$\xymatrix {C \ar[r]^{(\phi, \psi)} & C'}$ such that
$h' \circ \phi = h$. The \emph{germ relation} "$\,\germ\,$" \, between homotopies is the equivalence relation generated by "$\,\germ >\,$".
{So, $H \germ H'$ if they are connected by a zig-zag of cylinder maps}.
\end{definition}

\vspace{-1ex}

\emph{{Note that this definition of equivalence is \emph{syntactic} in the sense that it is a condition intrinsic to the 2-category
$\sr{A}$.}}
\begin{lemma} \label{germ==>adhoc}
    If $H$, $H': f  \Hpy{} g$ are two homotopies on $\sr{A}$ such that \mbox{$H \germ H' $}, then $[H]=[H']$ in $\HAo$.
\end{lemma}
\begin{proof}
    It suffices to prove the statement for the generators 
    \mbox{$H' \germ > H$.} Let $C=(W, Z, d_0, d_1, s, x)$, \mbox{$C'=(W', Z', d'_0, d'_1, s', x')$}, $H=(C, h)$, $H'=(C', h')$,  be such that $H' \germ > H$, and let
    $\xymatrix{C \ar[r]^{(\phi, \psi)} & C'}$ a cylinder map such that $h'\phi =h$.
 
        If $\mathscr{D}$ is a 2-category and $\sr{A} \mr{F}\mathscr{D}$ is a functor that sends weak equivalences into equivalences, since $F\psi Fs= Fs'F\phi$, then we have \mbox{$Fx=F\psi Fx=F\psi Fs \widehat{FC}=Fs'F\phi\widehat{FC}$, and $F\phi\widehat {FC}: Fd'_0 \Longrightarrow Fd'_1$.} For \mbox{uniqueness} of $\widehat{FC'}$, it results $F\phi\widehat{FC}=\widehat{FC'}$, and then from the equation $h\phi=h'$ we get $\widehat{FH'}=Fh'F\psi \widehat{FC}=Fh\widehat{FC}=\widehat{FH}.$
    \end{proof}

\begin{comentario} \label{germcoment1}
    The germ relation would allow to define another homotopy 
    \mbox{2-category}. We have verified that the vertical and horizontal compositions of classes of homotopies can be defined, and that all the requirements to have a \mbox{2-category} can be proved, except for the compatibility axiom that relates both compositions. We believe that this compatibility would not be valid, see more on this in Comment 
\ref{germcoment2}.
\end{comentario}

%

\section{Homotopy 2-categories of a model category.} \label{Loc_model}

We pass now to apply the results of Section \ref{hacheo} to develop a 2-dimensional version of Quillen's localizations of a model category 
$\sr{C}$, $(\mathcal{F}, \; \mathcal{C}o\!f, \; \mathcal{W})$, notation refers to \ref{MC}. We set $\Sigma = \cc{W}$ and we will consider the cases  $\sr{A} = \sr{C}$ and $\sr{A} = \sr{C}_{fc}$, the full subcategory of fibrant-cofibrant objects.

\subsection{The 2-localization of the category
\texorpdfstring{$\mathscr{C}_{fc}$}{}} \label{locationC_fc}

Using  Theorem \ref{2-loc-W-split} we see that for the 2-category $\Ho$ to be the localization of $\mathscr{C}$ at the class $\mathcal{W }$ it is enough that every weak equivalence be a composition of split weak equivalences. We will be able to prove this if we restrict ourselves to the full subcategory $\mathscr{C}_{fc}$ of fibrant-cofibrant objects, thus obtaining the 2-localization of $\mathscr{C}_{ fc}$ at the weak equivalences. {Note that $\mathscr{C}_{fc}$ is not a model category}. 

We take the following proposition from 
\cite[3.1.19]{S}:

\begin{proposition} \label{we_split_fc}
    Let $s: X \longrightarrow Y$ be a weak equivalence on a model category $\mathscr{C}$. If $X$ is fibrant and $Y$ is cofibrant, $\,s\,$ factors as a composition
\mbox{$X \mr{i} Z \mr{p} Y$ where $\,i\,$} is a section and $\,p\,$ a retraction, both weak equivalences. If $X$ and $Y$ are fibrant-cofibrant, the same happens with $Z$.

\end{proposition}

\begin{proof}
    Axiom M4 gives us a factorization $s=pi$, where $p$ is a fibration, $i$ a cofibration and one of the two is a weak equivalence. Since $s$ is a weak equivalence, by M5 we obtain that the three maps are weak equivalences.

   Now, since $X$ is a fibration, it follows that $i$ is a section: Indeed, since it is a trivial cofibration and $X\longrightarrow 1$ is a fibration, the lifting property guarantees the existence of the dotted arrow in the commutative diagram
 $\vcenter{
           \xymatrix@R=3ex
           {
            X \ar@{=}[r] \ar@{ >->}[d]_i | \circ & X \ar@{->>}[d]\\
            Z \ar@{-->}[ur] \ar[r] & 1,
           }
         }$
 , so $i$ has a left inverse.
Dually, $p$ is a retraction thanks to $Y$ being cofibrant.
If $X$ and $Y$ are fibrant-cofibrant, it is clear by axiom M3 that this is also the case for $Z$.
\end{proof}

From this proposition together with 
Proposition \ref{split==>loc} we have
\begin {proposition} \label{we==>eq_fc}
The inclusion
$i:\mathscr{C}_{fc}\longrightarrow \cc{H}o(\sr{C}_{fc})$ sends weak equivalences into equivalences. \cqd
\end{proposition}
 
{Theorem \ref{2-loc-W-split} gives us:}
 
\begin{theorem} \label{2_loc_fc}
    The inclusion
$i:\mathscr{C}_{fc}\longrightarrow \cc{H}o(\sr{C}_{fc})$ is the 2-localization, in the strict sense, of the subcategory $\mathscr{ C}_{fc}$ at the class
$\mathcal{W}_{fc} = \mathcal{W} \cap \sr{C}_{fc}$.
    
    Furthermore, the 2-functors $i^*: Hom_s(\cc{H}o(\sr{C}_{fc}), \mathscr{D})\longrightarrow Hom_{s}(\mathscr{ C}_{fc}, \mathscr{D})_+$ e $i^*: Hom_p(\cc{H}o(\sr{C}_{fc}), \mathscr{D})\longrightarrow Hom_{p}(\mathscr{C}_{fc}, \mathscr{D})_+$ are both 2-category isomorphisms.
\cqd
\end{theorem}
\emph{We remark that we shall see in \ref{category_Ho} below  that the 2-cells in 
$\cc{H}o(\sr{C}_{fc})$ are actually given by single classical Quillen's homotopies.}

\vspace{1ex}

We denote $\cc{H}o_{fc}(\sr{C})$ the full sub-2-category whose objects are the fibrant and cofibrant objects and whose 2-cells are classes of homotopies in
$\cc{H}o(\mathscr{C})$. We also have:
\begin{proposition}
The restriction $i: \mathscr{C} \supset \mathscr{C}_{fc}
\longrightarrow \cc{H}o_{fc}(\sr{C}) \subset
\cc{H}o(\mathscr{C})$ sends weak equivalences into equivalences. \cqd
\end{proposition}

{
 \begin{remark} \label{listo}
The 2-category $\cc{H}o_{fc}(\sr{C})$ does not coincide with $\cc{H}o(\sr{C}_{fc})$ because even in the case in which the cylinders can be taken in $\mathscr{C}_{fc}$ (as we will see later \ref{ready2}) the equivalence relation (Definition \ref{rel_adhoc}) is less fine since it quantifies only on 2-functors that admit an extension to $\cc{C}$. There is a 2-functor
\mbox{$j: \mathcal{H}o(\mathscr{C}_{fc}) \mr{} \cc{H}o_{fc}(\sr{C}) \subset \cc{H} o(\sr{C})$} which is the identity in objects and arrows, but in the
hom-categories
\mbox{$\mathcal{H}o(\mathscr{C}_{fc})(X,\,Y) \mr{} \fcHo(X,\,Y)$} would be neither full nor faithful. However, it turns out to be full, see Remark \ref{ready2}, but possibly not faithful in general. Assuming the normalised functorial factorization axiom, Definition
\ref{functorial}, these two categories coincide since then every 2-functor defined in $\mathscr{C}_{fc}$ extends to
$\mathscr{C}$.
\end{remark}
}                
\subsection{Quillen's homotopies as 2-cells} \label{Quillen_2cells}

To construct the homotopy category of a model category $\mathscr{C}$, Quillen introduces a notion of cylinder and of
left homotopy and shows that these determine an equivalence relation between the morphisms of the full subcategory of fibrant and cofibrant objects. Here we will see that the Quillen left homotopies determine the 2-cells of a 2-category on those objects, which coincides with the homotopy 2-category constructed in the subsection \ref{locationC_fc}. Quillen left homotopies are in particular left homotopies in the sense of the section \ref{hacheo},
so we can use the results of that section. 

\vspace{1ex}

Homotopies such that the arrow $s$ in their cylinders is a fibration play a relevant job, so it is convenient to have the following definition:

\begin{definition} \label{fibranthpy}
A cylinder as in definition \ref{cylindrical_homotopias} is fibrant when the arrow $s$ is a fibration. A fibrant homotopy is an homotopy with fibrant cylinder.
\end{definition}

We denote by $q$-cylinder and left $q$-homotopy the cylinders and the left homotopies of \cite{Qui}. We recall their definition:

\begin{definition} \label{q-cylinderdef}
    A \emph{q-cylinder} $C=(W, d_0, d_1, s)$ for an object $X$ in $\mathscr{C}$ is a factorization of the codiagonal
               $$
                \xymatrix{
                    X\amalg X \ar@/^2pc/[rrrr]^{\nabla_X} \ar[rr]^{\binom{d_0}{d_1}} && W \ar[rr]^{s} && X,
                }
               $$
    where $\binom{d_0}{d_1}$ is a cofibration and $s$ is a weak equivalence. Thus in particular $q$-cylinders are cylinders according to Definition  \ref{cylindrical_homotopias}   

\vspace{1ex}

A left q-homotopy $H = (C, h)$ from $f$ to $g$ is a homotopy according to Definition \ref{cylindrical_homotopias} whose cylinder is a q-cylinder, $hd_0=f$ and $hd_1=g$.

\vspace{-1ex}                
                
$$
\xymatrix@R=3ex
      {
        X\amalg X \hspace{1mm} \ar@{>->}[rr]^{\binom{d_0}{d_1}} 
                               \ar@/^2pc/[rrr]^{\binom{f}{g}}
                               \ar[dr]_{\nabla_X} 
      && W \ar[r]^{ h} \ar[dl]^{s}|\circ & Y, 
      \\
       & X
      }
$$
\noindent 1. Note that by axiom M4 any object $X$ has at least one fibrant  
\mbox{$q$-cylinder}.

\noindent 
2. Note that when X is cofibrant, since cofibrations are stable by pushout 
and composition, it follows that the object $W$ in any $q$-cylinder  is also cofibrant.
\end{definition}
 We work only with left $q$-homotopies, thus as before we will omit to write the word "left".

Our next task will be to show that for every homotopy $H$ between arrows of
$\mathscr{C}_{fc}$ there is a q-homotopy $H'$ that determines the same 2-cell in $\Ho$, and whose cylinder $C '$ is in $\mathscr{C}_{fc}$, Lemma \ref{lemma_3_steps}. 
In this proof we use Lemma \ref{germ==>adhoc}. Note that we will need cylinder maps in both directions. Before two lemmas of independent interest.

\begin{lemma} \label{exist_q_hpy}
Given any two objects $X$, $Y$, $f, g: X \longrightarrow Y$, and $H$  a fibrant homotopy from $f$ to 
$g$, there exists a fibrant $q$-homotopy $H'$ such that  
$H \germ H'$. 
\end{lemma}
\begin{proof}
We thanks M. Szyld who send us the simple proof here. Let 
\mbox{$C = (W, Z, d_0, d_1, s, x)$} be the cylinder of $H$, and take any fibrant $q$-cylinder $C' = (W', d'_0, d'_1, s')$ for $X$. We get the following diagram, where $t$ is given by axiom M2.
$$
\xymatrix@C=6ex
   {
    X \amalg X \ar@{=}[r] 
               \ar[d]^{\binom{d'_0}{d'_1}}
               \ar[dr]
               \ar@/_1.5pc/[dd]_{\nabla_X}
  & X \amalg X \ar[d]_{\binom{d_0}{d_1}}
               \ar@/^1.5pc/[dd]^{\binom{x}{x}}
 \\
    W' \ar@{-->}[r]^{t}
       \ar@{->>}[d]^{s'} |\circ
       \ar[dr]^{xs'}
  & W  \ar@{->>}[d]^{s} |\circ
 \\
    X \ar[r]^{x} 
  & Z
   }
$$
This determines a fibrant $q$-homotopy $h' = h\,t$ from $f$ to $g$ which is in the same class as $H$ since $(t, x)$ is a cylinder map which establishes the germ relation.
\end{proof}

\begin{lemma} \label{exist_fib_hpy}
Let $Y$ be a fibrant object, then for any homotopy $H$ between arrows with codomain $Y$, there exists a fibrant homotopy $H'$ such that  
\mbox{$H \germ H'$.}
\end{lemma}
\begin{proof}
Let $f \Hpy{H} g$ be
$\xymatrix@R=3.2ex@C=3.2ex
     {
      X \ar@<3pt>[rr]^{d_0} \ar@<-3pt>[rr]_{d_1}\ar[dr]_{x} 
   && W \ar[rr]^{h}\ar [dl]^{s}|\circ 
   && Y
   \\
    & Z 
   \\
      }$.    
Consider a factorization of $s$ given by
$\xymatrix@R=3.2ex@C=3.2ex
    { 
     W \ar[rr]^{s}| \circ \ar@{>->}[dr]_{j}|\circ && Z\\
   & W' \ar@{->>}[ur]_{s'} |\circ
     }$,
and a map $h':W'\longrightarrow Y$ making the diagram commute
$\xymatrix
    { 
     W\ar[r]^{h} \ar@{ >->}[d]_(.4){j}|(.4)\circ 
   & Y \ar@{->>} [d] 
  \\
     W' \ar@{-->}[ur]^{h'} \ar[r] & 1  
     }$.
 Taking \mbox{$d'_0:=jd_0$} and $d'_1:=jd_1$ we have 
$\xymatrix@R=3.2ex@C=3.2ex
    {
     X \ar@<3pt>[rr]^{d'_0}\ar@<-3pt>[rr]_{d'_1}\ar[dr]_{x}   
  && W'\ar[rr]^{h'}\ar@{->>}[dl]^{s'}|\circ 
  && Y
  \\
   & Z 
     }$,
 which is a fibrant homotopy from $f$ to $g$  in the same class as $H$ since $(j, id_Z)$ is a cylinder map is a cylinder map which establishes the germ relation. 
\end{proof}
\begin{lemma} \label{lemma_3_steps}
Let $f, g: X \longrightarrow Y$ with $X$ and $Y$ both fibrant and cofibrant, and let $H$ be a homotopy from $f$ to $g$ in $\mathscr{C}$. Then there exists $H'$, a \mbox{$q$-homotopy} in $\mathscr{C}_{fc}$ such that $[H]=[H']$ in $\cc{H}o(\sr{C })$. Moreover, actually $H \germ H'$ holds.
\end{lemma}
\begin{proof}
Assuming only that $X$ is cofibrant and $Y$ fibrant we use the previous lemmas in turn, and obtain the desired $q$-homotopy $H'$. It remains to see that $W'$ in the cylinder of $H'$ is inside $\mathscr{C}_{fc}$, for this we need that $X$ be also fibrant. Indeed, $W'$  
is cofibrant by $2.\!$ in Definition \ref{q-cylinderdef}, and since 
$W' \mr{s'} X$ is a fibration, it follows that $W'$ is also fibrant. Finally, Lemma \ref{germ==>adhoc} finishes the proof.  
 \end{proof}    
\begin{remark} \label{ready2}
Note that this Lemma implies in particular that for the \mbox{2-functor} $\mathcal{H}o(\mathscr{C}_{fc}) \mr{} \fcHo$ of the Remark \ref{listo}, the functors between the hom-categories
\mbox{$\mathcal{H}o(\mathscr{C}_{fc})(X,\,Y) \mr{} \fcHo(X,\,Y)$} are full.
\end{remark}

\begin{comentario} \label{rcompoq}
We refer to \ref{horiz_composition_def}. The horizontal composition of classes of $q$-homotopies with arrows on the right $r[H]$ can be defined as $r[H]=[rH]$ since if $H$ is a 
\mbox{$q$-homotopy,} so it is $rH$. It does not occur the same for arrows on the left. 
Lemma \ref{lemma_3_steps} tells us that when the objects are fibrant and cofibrant, the composite $[H]l$ can be defined as $[H]l = [H']$ where $H'$ is a $q$-homotopy such that $[Hl] = [H']$.
 Quillen solves the problem posed by the definition of the left $q$-homotopy $[H]l$ in a different way. Enhancing Quillen's argument to our 2-dimensionl context, take a right \mbox{$q$-homotopy} $K$ such that $[H] = [K]$, see Proposition \ref{lefttoright}. Then $Kl$ is a right $q$-homotopy, which in turn has a left $q$-homotopy $H'$ such that 
$[Kl] = [H']$, then define $[H]l = [H']$.
\end{comentario}

%


\begin{sinnada} \label{verticalcomposition}
{\bf Vertical composition of q-homotopies\,.}
\end{sinnada}

To construct the homotopy category in \cite{Qui} Quillen shows that the (left) homotopies define an equivalence relation between the morphisms of the full subcategory of cofibrant objects. To prove the transitivity (we refer to the notations above), given $q$-homotopies $f \Hpy{H} g \Hpy{H'} l$, $W \mr{h} Y$, $W' \mr{h'} Y$,
he considers the $q$-cylinder $C''$ determined by the pushout $W''$ 
of $d_1$ and $d'_0$, then the universal property of the pushout yields the $q$-homotopy $f \Hpy{H''} l$, $W'' \mr{h''} Y$ that establishes the transitivity.  
%
%

\vspace{1ex}

In the category of topological spaces, this pushout is precisely the cylinder that is obtained by gluing the bottom of the first cylinder with the top of the second:

$
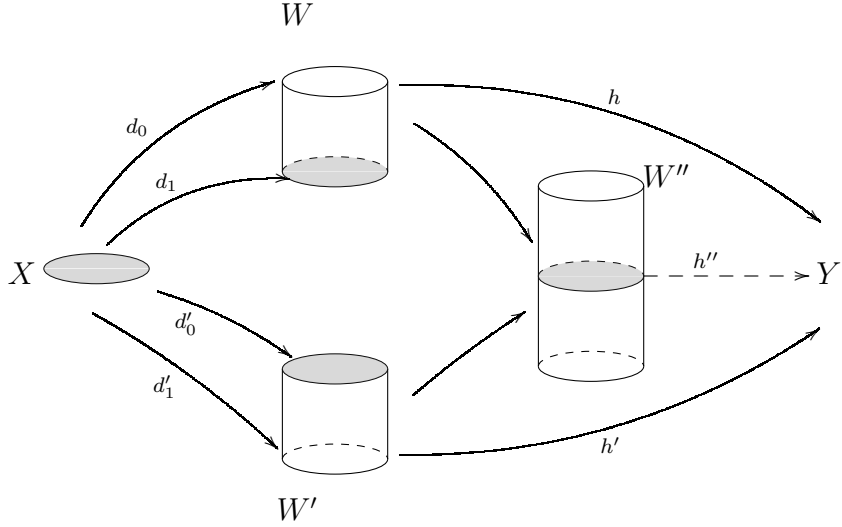
\xymatrix@C=1.5pc@R=1.7pc{
&& W\\
&&\\
&& \save []+<0.5cm, 0.7cm>*{\begin{tikzpicture}
        \draw (0,0) ellipse (0.7 and 0.2);
        \draw (-0.7,0) -- (-0.7,-1.2);
        \draw (-0.7,-1.2) arc (180:360:0.7 and 0.2);
        \draw [dashed] (-0.7,-1.2) arc (180:360:0.7 and -0.2);
        \draw (0.7,-1.2) -- (0.7,0);  
        \fill [gray,opacity=0.3] (-0.7,-1.2) arc (180:360:0.7 and 0.2) -- (-0.7,-1.2) arc (180:360:0.7 and -0.2);
    \end{tikzpicture}} \ar@<15pt>@/^0.6pc/[drrr] \ar@<15pt>@/^1.7pc/[drrrrrrr]^h \restore &&&&& \hspace{1mm} W'' \\
X\hspace{1mm} \begin{tikzpicture}
        \draw (0,0) ellipse (0.7 and 0.2);
        \fill [gray,opacity=0.3] (-0.7,0) arc (180:360:0.7 and 0.2) -- (-0.7,0) arc (180:360:0.7 and -0.2);
  \end{tikzpicture} \ar@/^1pc/[urr]^{d_1}
                  \ar@<9pt>@/^1pc/[uurr]^{d_0}
                  \ar@/^0.5pc/[drr]_{d'_0}
                  \ar@<-10pt>@/^0.3pc/[ddrr]_{d'_1} &&&&&
 \save []+<1cm, 0cm>*{\hspace{2mm} \begin{tikzpicture}
        \draw (0,0) ellipse (0.7 and 0.2);
        \draw (-0.7,0) -- (-0.7,-1.2);
        \draw (-0.7,-1.2) arc (180:360:0.7 and 0.2);
        \draw [dashed] (-0.7,-1.2) arc (180:360:0.7 and -0.2);
        \draw (0.7,-1.2) -- (0.7,0);
        \draw (-0.7,-1.2) -- (-0.7,-2.4);
        \draw (-0.7,-2.4) arc (180:360:0.7 and 0.2);
        \draw [dashed] (-0.7,-2.4) arc (180:360:0.7 and -0.2);
        \draw (0.7,-2.4) -- (0.7,-1.2);
        \fill [gray,opacity=0.3] (-0.7,-1.2) arc (180:360:0.7 and 0.2) -- (-0.7,-1.2) arc (180:360:0.7 and -0.2);
    \end{tikzpicture}} \ar@{-->}[rrrr]^{h''} \restore &&&& Y \\
&& \save []+<0.5cm, -0.7cm>* {\begin{tikzpicture}
        \draw (0,0) ellipse (0.7 and 0.2);
        \draw (-0.7,0) -- (-0.7,-1.2);
        \draw (-0.7,-1.2) arc (180:360:0.7 and 0.2);
        \draw [dashed] (-0.7,-1.2) arc (180:360:0.7 and -0.2);
        \draw (0.7,-1.2) -- (0.7,0);  
        \fill [gray,opacity=0.3] (0,0) ellipse (0.7 and 0.2);
    \end{tikzpicture}} \ar@<-15pt>@/^0.2pc/[urrr] \ar@<-15pt>@/_1.5pc/[urrrrrrr]_{h'} \restore \\                  
&&\\
&& W' \\}$ 

%

\vspace{2ex}

We recall now Quillen's development of these ideas with the necessary precision needed to establish the vertical composition of arrows in the homotopy 2-category, Lemma \ref{hpycompose}.
\begin{lemma}[\cite{Qui} Ch.I, $\S$1, Lemma 2] \label{lemma2_quillen}
Let $X$ be a cofibrant object and let \mbox{$C=(W, d_0, d_1, s)$} be $q$-cylinder for $X$, then $d_0$ and $d_1$ are trivial cofibrations.  
\end{lemma}
%
%
%
This is used in the proof of the following:
\begin{lemma}[\cite{Qui}, Ch.I, $\S$1, Lemmas 3,\:4)] \label{comp_v}
    
Let $f, g, l : X \mr{} Y$, and let 
$H: f \Hpy{} g$, $H': g \Hpy{} l$ two q-homotopies.  
If $X$ is cofibrant, then there is a q-homotopy 
$H'': f \Hpy{} l$ 
as the result of the following construction:

\vspace{1ex}

Let $C=(W, d_0, d_1, s)$ and $C'=(W', d_0', d_1', s')$ be cylinders for $H$ and $H'$ respectively, and let $W''$ be the pushout of $d_{1}$ and $d'_{0}$ with inclusions $\alpha, \; \beta$ as indicated in the diagram \ref{compo2} below:
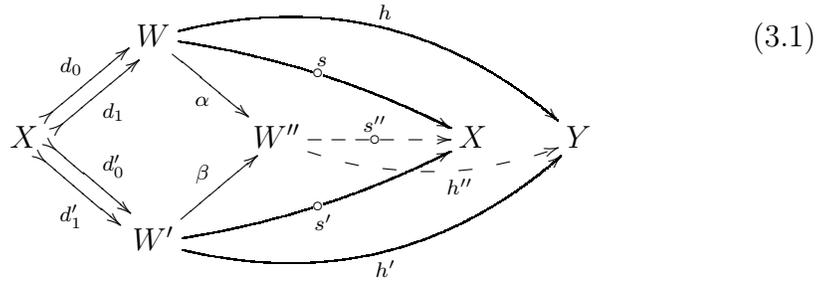
\begin{equation} \label{compo2}
    \xymatrix{
            & W \ar[dr]_{\alpha}
                 \ar@/^/[drrr]^{s}|\circ \ar@/^2pc/[rrrrd]^{h}\\
        X \hspace{2mm} \ar@{>->}@<3pt>[ur]^{d_0}
          \ar@{>->}@<-3pt>[ur]_{d_1}
          \ar@{>->}@<3pt>[dr]^{d'_0}
          \ar@{>->}@<-3pt>[dr]_{d'_1} && W'' \ar@{-->}[rr]^{s''}|\circ \ar@{ -->}@/_1pc/[rrr]_(.6){h''} && X & Y\\
            & W' \ar[ur]^{\beta}
                 \ar@/_/[urrr]_{s'}|\circ \ar@/_2pc/[rrrru]_{h'}\\}
\end{equation}
We define a cylinder 
$C''=(W'', d_0'', d_1'', s'')$, $d_0''=\alpha d_0$, $d_1''=d'_1\beta$,    
\mbox{$s'' \alpha = s, \; s'' \beta = s'$,} which is indeed a $q$-cylinder. 
Then, since $hd_1=g=h'd'_0$ we have an arrow $h''$,  
$h'' \alpha = h, \; h'' \beta = h'$ which determines a $q$-homotopy 
$H'': f \Hpy{} l$ with cylinder $C''$. 
\end{lemma}
\begin{remark} \label{comp_v_2}
If $X$ and $Y$ are fibrant-cofibrant, in Lemma \ref{comp_v} we may assume that the cylinder $C''$ is inside $\sr{C}_{fc}$.  To see that we may assume $W''$ fibrant, we proceed as in  \ref{exist_fib_hpy}, and observe that if $H$ is a $q$-homotopy, then $H'$ still is a $q$-homotopy. That $W''$ is cofibrant follows by Lemma \ref{lemma2_quillen} or 
\mbox{$2.\!$ in Definition \ref{q-cylinderdef}.}  \cqd
\end{remark}


\begin{lemma} \label{hpycompose}
    Given $f, g, l: X \mr{} Y$ in $\mathscr{C}_{fc}$, and two composable \mbox{$q$-homotopies} inside  $\sr{C}_{fc}$, 
    $H: f \Hpy{} g $ and $H': g \Hpy{} l$, there exists a 
    \mbox{$q$-homotopy} $H'': f \Hpy{} l$ inside  $\sr{C}_{fc}$ such that
 $[H'']=[H', H]$ in $\cc{H}o(\sr{C}_{fc})$.
\end{lemma}



\begin{proof}
Consider the homotopy $H''=(C'', h'')$ in Lemma \ref{comp_v} which by Remark \ref{comp_v_2} is inside $\sr{C}_{fc}$. Let's see that 
$\widehat{FH''}=\widehat{FH'} \circ \widehat{FH}$ for every 2-functor
$F: \mathscr{C}_{fc} \longrightarrow \mathscr{D}$ which sends weak equivalences into equivalences. Consider diagram \ref{compo2}:
 
 \vspace{1ex}
    
Since $\widehat{FH'}\circ \widehat{FH} = 
Fh'\widehat{FC'} \circ Fh\widehat{FC} =
Fh''F\beta\widehat{FC'} \circ Fh''F\alpha \widehat{FC} = Fh'' (F\beta \widehat{FC'} \circ F\alpha \widehat{FC})$, 
it suffices to see that 
$F\beta \widehat{FC'} \circ F\alpha \widehat{FC} =\widehat{FC''}$. 
From the equations $s=s''\alpha$ and $s'=s''\beta$ it follows that $id_{FX}= id_{FX} \circ id_{FX}= 
Fs'' F\beta \widehat{FC'} \circ Fs''F\alpha \widehat{FC} =
Fs''(F\beta \widehat{FC'} \circ F\alpha \widehat{FC})$, so that 
\mbox{$F\beta \widehat{FC'} \circ F \alpha \widehat{FC} = \widehat{FC''}$.}
\end{proof}

With what we have seen so far we can ensure that when we restrict ourselves to the subcategory $\mathscr{C}_{fc}$ there is a correspondence between the classes of $q$-homotopies and the classes of finite sequences of composable $q$-homotopies, and this together with Lemma \ref{lemma_3_steps} give us the following proposition:

{
\begin{corollary} \label{category_Ho}
 The 2-category $\cc{H}o(\sr{C}_{fc})$ in Theorem \ref{2_loc_fc} coincides with the one whose 2-cells are the classes of $q$-homotopies.
\end{corollary}
}

\vspace{2ex}

\begin{sinnada} \label{Quillen_loc}
{\bf Classic Quillen's localization of $\sr{C}_{fc}$\,.}
\end{sinnada}

In this item we will see how to obtain the homotopy category of
$\mathscr{C}_{fc}$ defined by Quillen from the 2-category $\cc{H}o(\sr{C}_{fc})$.


We will obtain Quillen's localization by applying the functor of connected components \mbox{$\pi _0 : 2$-$Cat \mr{} Cat$.}
Recall that if $\sr{D}$ is a 2-category, for every pair of objects $X, \: Y$, $(\pi_0\sr{D})(X, \: Y)$ is the set of connected components of the category $\sr{D}(X, \: Y)$.
Remember that $\pi_0$ is the left adjoint of the functor 
\mbox{$d: Cat \longrightarrow 2$-$Cat$} which associates each category $\mathscr{X}$ with itself seen as a discrete 2-category.
Abusing notation we have an obvious 2-functor
\mbox{$\sr{D} \mr{\pi_0} \pi_0\sr{D}$} which is universal with respect to 2-functors of $\sr{D}$ with values in a 1-category. We have:
    


\emph{(1) The functor $Hom(\pi_0\mathscr{D}, \mathscr{X}) \mr{\pi_0^*} Hom_p(\mathscr{D}, \mathscr{X})$ is an isomorphism of \mbox{categories.}}

%
The composition of $\pi_0$ with $\, i\,$:
$\mathscr{C}_{fc} \mr{i} \cc{H}o(\sr{C}_{fc})
                        \mr{\pi_0} \pi_0\cc{H}o(\sr{C}_{fc})$,
sends weak equivalences into isomorphisms since $i(\mathcal{W})\subseteq Equiv\catHofc{C}$, and, from the way in which $\pi_0$ is defined in the 2-cells, it sends equivalences into isomorphisms. Then Theorem \ref{2_loc_fc} together with (1) above show that the composite functor $\pi_0 \,i$ is the Quillen localization.

\begin{theorem} \label{Quillenloc1}
 The functor defined by the composition
 $\mathscr{C}_{fc} \mr{\pi_0 i} \pi_0\cc{H}o(\sr{C}_{fc})$
is the localization of $\mathscr{C}_{fc}$ with respect to the class
$\mathcal{W}$. The precomposition
$$
(\pi_0i)^* = i^* \pi_0^* :Hom(\pi_0\cc{H}o(\sr{C}_{fc}), \mathscr{X})\longrightarrow Hom(\mathscr{ C}_{fc}, \mathscr{X})_+
$$
is a isomorphism of categories for every category $\mathscr{X}$.
\cqd
\end{theorem}
{Beyond the formality of the proof, we can understand how it works by looking at the following diagram:}

\vspace{-1.5ex}

$$
\xymatrix@C=3ex@R=3ex
    {
     \mathscr{C}_{fc} \ar@{^{(}->}[rr]^{i}
                      \ar[ddr]_(.4){F}
  && \cc{H}o(\sr{C}_{fc}) \ar[rr]^{\pi_0}
         \ar@{-->}[ddl]^(.4){\!\exists !}
  && \pi_0\cc{H}o(\sr{C}_{fc}), \ar@{-->}[ddlll]^(.4){\!\exists !}
 \\
 \\
   & {\mathscr{X} } \\}
$$ 

\vspace{-4ex}

\cqd

\vspace{2ex}

\begin{remark} \label{pi0Hfc=pi0fcH}
Note than by Remark \ref{ready2} the hom categories of the 
\mbox{2-categories} 
$\fcHo$ and $\Hofc$ have the same connected components.
\end{remark} 
    
\subsection{The 2-localization of the category \texorpdfstring{$\mathscr{C}$}{}} \label{2_loc_srC}
%
%
%
We want to define a functor \mbox{$\mathscr{C} \longrightarrow \mathscr{C}_{fc}$} and take the composition with \mbox{$i: \mathscr{C}_{fc} \longrightarrow \Hofc$.} We will prove that this 2-functor, which we will call
\mbox{$q:\mathscr{C} \longrightarrow \Hofc $,} is the 2-localization of the category $\mathscr{C}$ with respect to the class $\mathcal{W}$, which is our main goal . For this we will consider the full subcategories $\mathscr{C}_f$ and $\mathscr{C}_c$ of fibrant and cofibrant objects respectively, and two assignments $R: \mathscr{C} \longrightarrow \mathscr{C }_f$ and $Q: \mathscr{C} \longrightarrow \mathscr{C}_c$ which can be constructed from the model category axioms in Definition \ref{model_cat},
but it is clear on inspection of the construction that these assignments are not necessarily functorial.

We adopt here an ad hoc solution to this problem by requiring that the factorization of the axiom M4 be functorial, a requirement frequently used in the post-Quillen literature.
Assuming that the model structure of $\mathscr{C}$ admits a functorial factorization, the arrows $\xymatrix{ \mathscr{C} \ar[r]^{Q} &\mathscr{C}_c \ar[r]^{R} & \mathscr{C}_{fc} }$ will effectively determine a functor which together with what we saw in subsection \ref{locationC_fc} for the subcategory $\mathscr{C}_{fc }$, will allow us to conclude the localization theorem. We remark that also the fundamental theorem \ref{hom_iso} in subsection 
\ref{fundamental} will be necessary.

\subsubsection{Fibrant and cofibrant replacements} \label{fibrant-cofibrant}

From now on we will work on a model category $\mathscr{C}$ in which the factorizations of axiom M4 are \emph{normal functorial} in the sense of Definition \ref{functorial}.

\begin{definition} \label{cofibrantreplacement}
    Let $\mathscr{C}_c \subseteq \mathscr{C}$ be the subcategory of cofibrant objects. We define a functor  
    $Q: \mathscr{C} \longrightarrow \mathscr{C}$ that takes its values in the subcategory $\sr{C}_c$ and such that
if $X$ is already cofibrant, $QX = X$ and $p_X = id_X$, 
and such that if $f$ is a weak equivalence, so is $Qf$. We do this as follows:
    
Let $F:\overrightarrow{\mathscr{C}} \longrightarrow \mathscr{C}$ be the functor of the normal functorial realisation associated with the factorization, 
 we set $QX=F(0 \rightarrow X)$ and 
  $Qf= F(id_0, f)$. Then by the functoriality of $F$ it follows that $Q$ is also a functor:
    \begin{center}
        $\xymatrix@R=4ex@C=4ex{
            0 \ar[rr] \ar[d] && 0\ar[d]\\
                QX \ar[rr]|{id_{QX}} \ar@{->>}[d]_{p_X}|\circ && 
                QX \ar@{->>}[d]^{p_X}|\circ  & \hspace{3mm}\text{and}\\
                X \ar[rr]_{id_X} && Y\\
        }$
        \hspace{6mm}
        $\xymatrix@R=4ex@C=4ex{
            0\ar[rr] \ar[d] && 0 \ar[rr] \ar[d] && 0 \ar[d] \\
            QX \ar@/^1.5pc/[rrrr]|(.4){Q(gf)} \ar[rr]_{Qf} \ar@{->>}[d]_{p_X}|\circ && QY \ar[rr]_{Qg} \ar@{->>}[d]_{p_Y}|\circ && QZ \ar@{->>}[d]^{p_Z}|\circ \\
            X \ar[rr]_{f} && Y\ar[rr]_{g} && Z\\
        }$
    \end{center} 
By axiom M5, if $f$ is a weak equivalence, then so is $Qf$.

Dually, we also obtain a functor $R:\mathscr{C} \longrightarrow \mathscr{C}$ with $RX$ a fibrant object, and a trivial cofibration $v_X: X \longrightarrow RX$, this time factoring   the morphism $X \longrightarrow 1$,  and for $f:X \longrightarrow Y$, an arrow $Rf: RX \longrightarrow RY$ satisfying \mbox{$Rfv_X=v_Yf$.} Also, if $X$ is already fibrant, $RX = X$ and $v_X = id_X$, and if $f$ is a weak equivalence, then so is $Rf$.
\end{definition}
\begin{remark}
The functors $Q$ and $R$ are known as \emph{cofibrant replacement} and \mbox{\emph{fibrant replacement},} respectively, and determine by composition a functor \mbox{$RQ: \sr{C} \mr{} \sr{C}$} with values in the subcategory
$\sr{C}_{fc}$. \cqd
\end{remark}
\begin{remark} \label{assignment}
We have natural transformations $p:Q \Longrightarrow Id$ and 
    \mbox{$v: Id \Longrightarrow R$} defined by $p_X$ and $v_X$, respectively.
    From the definitions of $Q$ and $R$ in the morphisms we have the naturality equations for $p$ and $v$.
{
Combining both, we have
\begin{center}
$
\xymatrix@C=4ex
       {
        go
       &
        Q \ar@2{->}[l]_{p}
          \ar@2{->}[r]^{vQ}
       &
        RQ
       }
$,
and for each $X \in \sr{C}$,
$
\xymatrix@C=4ex
       {
        X
       &
        QX \ar@{->}[l]_{p_X}
          \ar@{->}[r]^{v_{QX}}
       &
        RQX
       }
$,
\end{center}
 where $p_X$ and $v_{QX}$ are weak equivalences. \cqd
}
\end{remark}
\begin{proposition} \label{functor_q}
We have the following commutative diagram whose construction  we explain step by step in the proof below:
$$
\xymatrix@C=8ex
   {
    \sr{C} \ar[r]^{RQ}
           \ar[dr]^{r}
           \ar[ddr]^{q}
           \ar[ddd]^{i}
  & \sr{C} \ar@/^3.8pc/[ddd]^{i}
  \\
    {} & \sr{C}_{fc} \ar[d]^{i'}
                     \ar[u]_{j}
                     \ar@/^2.2pc/[dd]^{ij}
  \\
    {} & \cc{H}o(\sr{C}_{fc}) \ar[d]^{\ol{j}}
  \\
    \cc{H}o(\sr{C}) \ar[ru]^{\ol{r}}
                    \ar[r]^{\ol{RQ}}
  & \cc{H}o(\sr{C})
   }
$$
\end{proposition}
\begin{proof} ${}$

{\bf 1.} Restricting $R$ to the subcategory $\mathscr{C}_c$ and precomposing with $Q$ we have a functor $r: \mathscr{C} \longrightarrow \mathscr{C}_{fc}$ , where "$j$"\, denotes the full inclusion $\sr{C}_{fc} \subset \sr{C}$. Note that $r$ is the correstriction of the composition $RQ$, and that $r\,j\, = id$. 

{\bf 2.} We define $q:\mathscr{C} \longrightarrow \Hofc$ to be the composition, $q = i'\,r$.

{\bf 3.} From Definition \ref{cofibrantreplacement} and Proposition \ref{we==>eq_fc} we know that $\,q\,$ sends weak equivalences into equivalences, then from Proposition \ref{pu_i} there exists a unique 2-functor $\,\ol{r}\,$ such that $\ol{r} \, i = q = i'\,r$.

{\bf 4.} We consider the restriction $i\,j$ of $\,i\,$, and again by Proposition \ref{we==>eq_fc} we know that $i\,j$ sends weak equivalences into equivalences, and from Theorem
\ref{2_loc_fc} we get a unique 2-functor $\,\ol{j}\,$ such that $\ol{j}\, i' = i\,j$. Precomposing with $i'$ we compute 
$\ol{r}\,\ol{j}\,i' = \ol{r}\,i\,j =i'\,r\,j= i'$, then Theorem \ref{2_loc_fc} gives us
$\ol{r}\,\ol{j} = id$.

{\bf 5.} Finally we have a 2-functor $\ol{RQ} = \ol{j}\:\ol{r}$ wich  satisfies 
$\ol{RQ}\,i = i\, RQ$.
\end{proof}
\begin{remark} \label{noQnoRinHo}
We remark that there are not 2-functors $\ol{Q}$ and  $\ol{R}$ of cofibrant and fibrant replacement defined on the homotopy category $\cc{H}o(\sr{C} )$,
$\ol{RQ}$ is not a composite $\ol{R} \circ \ol{Q}$.
\end{remark}

\subsubsection{The 2-localization theorem} \label{2location}

We want to see that the functor $q:\mathscr{C} \longrightarrow \Hofc$ determines the \mbox{2-localization} of $\sr{C}$ at the weak equivalences, that is, that we have a 2-category pseudoequivalence
$$
q^*: Hom_p(\cc{H}o(\sr{C}_{fc}), \mathscr{D}) \mr{}
             Hom_{p}(\mathscr{C}, \mathscr{D})_+
$$
for every 2-category $\mathscr{D}$.

\vspace{1ex}

Since $q = \ol{r} \, i$, to prove that $q^\ast$ is a pseudoequivalence it would be enough to see that 
$\ol{r}^*: Hom_p(\cc{H}o(\sr{C}_{fc}), \sr{D}) \longrightarrow Hom_{p}(\cc{H}o(\sr{C}), \mathscr{D})_+$
as  well as \mbox{$i^*: Hom_p(\cc{H}o(\sr{C}), \sr{D})_+ \longrightarrow Hom_{p}(\mathscr{C}, \mathscr{D })_+$}
are both 2-category pseudoequivalences. 

Concerning $i$, we already saw in Theorem \ref{hom_iso} that in fact  
it is an isomorphism. Now:

\begin{theorem} \label{pseudoequivalence}
    The 2-functors
$$
\xymatrix
    {
     Hom_p(\cc{H}o(\sr{C}_{fc}), \sr{D})
                     \ar@<5pt>[rr]^{\bar{r}^*}
                     \ar@<-5pt>@{<-}[rr]_{\bar{j}^*}
   &&
     Hom_{p}(\Ho, \mathscr{D})_+
    }
$$
determine a 2-category pseudoequivalence.
\end{theorem}

%
\begin{proof}
We have
$\ol{j}^*\,\ol{r}^* = (\ol{r}\,\ol{j})^* = id^* = id$ (Proposition \ref{functor_q}, item 4), then it only remains to establish an equivalence of
$\ol{r}^*\,\ol{j}^* = (\ol{j}\:\ol{r})^*$ with $id$.

\vspace{1ex}

Let $F: \Ho \mr{} \sr{D}$ be such that it sends weak equivalences into equivalences. Applying the 2-functor
$\sr{C} \mr{i} \Ho$ followed by $F$ in the Remark \ref{assignment}, we have 2-natural transformations
$ \xymatrix@C=4ex
       {
        fi
       &
        FiQ \ar@2{->}[l]_{Fip}
          \ar@2{->}[r]^{FivQ}
       &
        FiRQ
       }
$. We know that for each object $X$ in $\mathscr{C}$, both $v_X$ and $p_X$ are weak equivalences (Remark \ref{assignment}), so $(Fip)_X$ and $( FivQ)_X$ are equivalences in $\sr{D}$.
 Let $(Fip)'$ be a pseudonatural quasi-inverse of $Fip$ obtained by  Proposition \ref{prop_infinite}. We define
$\eta_F: Fi \Mr{} FiRQ$ as the composition
$FivQ \circ (Fip)'$. But $iRQ = \ol{RQ}\,i$ and
$\ol{RQ} = \ol{j}\:\ol{r}$ (Proposition \ref{functor_q}, item 5), then we have
$\eta_F: Fi \Mr{} F \,\ol{j}\:\ol{r}\,i$. Finally Theorem
\ref{hom_iso} gives us an equivalence
$\eta_F: F \Mr{} F \,\ol{j}\:\ol{r} = \ol{r}^*\,\ol{j}^* (F) $ in
$Hom_{p}(\Ho, \mathscr{D})_+$

\vspace{1ex}

In turn, if $\eta$ were pseudonatural in the variable $F$, then
it would be the equivalence sought, since by the above, each component $\eta_F$ is so. Let's see that
$\eta$ is indeed pseudo-natural.

\vspace{1ex}
     
     In the objects of $Hom_{p}(\Ho, \mathscr{D})_+$, we have $\eta$ already defined. Now, given an arrow in $Hom_{p}(\Ho, \mathscr{D})_+$, that is a pseudonatural transformation
$\sigma: F \Longrightarrow G$, we want to define an invertible modification
     \mbox{$\eta_{\sigma}:\ol{r}^*\,\ol{j}^*(\sigma)\circ \eta_F \longrightarrow \eta_G \circ \sigma$}\,:
$$
\xymatrix@R=0.5pc@C=1.5pc
     {
      F \ar@{=>}[rr]^{\eta_F}
        \ar@{=>}[dd]_{\sigma}
    &&
      \ol{r}^*\ol{j}^*(F)
        \ar@{=>}[dd]^{\ol{r}^*\ol{j}^*(\sigma)}
        \ar@{}[ddll]|{\rotatebox[origin=c]{45}
                         {$\leftarrow$} \eta_{\sigma}}
   \\ \\
      G \ar@{=>}[rr]_{\eta_G}
    &&
      \ol{r}^*\ol{j}^*(G)
    \\
      }
$$

\vspace{1ex}

For each $X$ in $\sr{C}$ we have the following diagram that allows us to define $\eta_\sigma$ pointwise,
recall $\ol{r}^*\,\ol{j}^* (F) = F\,\ol{j}\:\ol{r}$ and
\ref{functor_q}, 5.:
$$
\xymatrix@R=3pc@C=3pc
      {
       FiX \ar[r]^{F'(ip_X)}
          \ar@/^2.3pc/[rr]^{(\eta_F)_X}
          \ar[d]_{\sigma_{iX}}
          \ar@{}[dr]|{\rotatebox[origin=c]{45}{$\Leftarrow$}
                                       \sigma '_{(ip_X)}}
      &
       FiQX
           \ar[r]^{F(iv_{QX})}
           \ar[d]^(.4){\sigma_{iQX}}
      &
       FiRQX
          \ar[d]^{\sigma_{iRQX}}
          \ar@{}[dl]|{\rotatebox[origin=c]{45}{$\Leftarrow$}
                                       \sigma_{(iv_{QX})}}
     \\
       GiX \ar[r]^{{G'(ip_X)}}
           \ar@/_2.3pc/[rr]^{(\eta_G)_X}
     & GiQX \ar[r]^{G(iv_{QX})}
     & GiRQX
       }
$$
The 2-cell $\sigma_{(iv_{QX})}$ in the right square is the inverse of the pseudo-natural structure of $\sigma$.  In the left square we denote $F'(ip_X)$, $G'(ip_X)$  the quasi-inverses of $F(p_{iX})$, $G(p_{iX})$. We have
$id \Mr{\alpha_G} G'(ip_X)\,G(ip_X)$ and
\mbox{$F(ip_X)\,F'(ip_X) \Mr{\beta_F} id$.}
The 2-cell $\sigma'_{(ip_X)}$ is defined by the following composition:

$
\sigma_{iQX}\,F'(ip_X)
                \Mrr{8}{\alpha_G\,\Box\,\Box\,}
G'(ip_X)\,G(ip_X)\,\sigma_{iQX}\,F'(ip_X)
                \Mrr{12}{\Box\,\sigma_{(ip_X)}\,\Box}
$

\vspace{1ex}

$
\hspace{22ex}
G'(ip_X)\,\sigma_{iX}\,F(ip_X)\,F'(ip_X)
               \Mrr{8}{\Box\,\Box\,\beta_F}
G'(ip_X)\,\sigma_{iX}
$

\vspace{1ex}

\noindent Composing this pasting diagram the desired modification is obtained.
\end{proof}

This finishes the proof of the localization theorem:
                     
\begin{theorem} \label{2-localization}
    Given a model category $\mathscr{C}$, the 2-functor
\mbox{$q: \mathscr{C} \longrightarrow \cc{H}o(\sr{C}_{fc})$} in  Proposition \ref{functor_q}, item 2., is the 2-localization of $ \mathscr{C}$ at the class $\mathcal{W}$, that is, it sends the elements of $\mathcal{W}$ into equivalences, and the 2-functor $q^* : Hom_p(\cc{H}o(\sr{C}_{fc}), \mathscr{D}) \longrightarrow Hom_{p}(\sr{C}, \mathscr{D})_+$ of precomposition with $q$, is a 2-category pseudoequivalence (see \ref{pseudoequiv}) for every 2-category $\mathscr{D}$. 
\end{theorem}
It is important to recall that although it is not necessary for the proof, in this theorem  $\Hofc = \fcHo$ since we are assuming functorial factorization, see \ref{listo}.

\begin{sinnada} 
{\bf Local smallness.}
We finish by observing that in Corollary \ref{loc_small} in the appendix it is established that the hom categories of the homotopy \mbox{2-categories} in the 2-localizations are locally small.
\end{sinnada}

\begin{sinnada} \label{Quillen_loc_2}
{\bf Classic Quillen's localization of $\sr{C}$.}
\end{sinnada}
 Applying the functor of connected components $\pi_0$ we conclude a result analogous to that of the item \ref{Quillen_loc}. More precisely, what we have is that the composition
$\mathscr{C} \mr{q} \cc{H}o(\sr{C}_{fc}) \mr{\pi_0} \pi_0\cc{H}o(\sr{C}_{ fc})$
is the localization of the category $\mathscr{C}$ with at the class $\cc{W}$ of weak equivalence, in the sense that its precomposition induces an equivalence of categories. So it is not exactly Quillen's localization which establishes an isomorphism, but it is equivalent to it.
From Theorem \ref{2-localization} as we did in the item \ref{Quillen_loc} to obtain Theorem \ref{Quillenloc1}, and observing certain details, we now obtain
\begin{theorem} \label{Quillen_loc'}
 The functor defined by the composition
 $\mathscr{C} \mr{\pi_0 q} \pi_0\cc{H}o(\sr{C}_{fc})$
is the localization of $\mathscr{C}$ with respect to the class
$\mathcal{W}$ in the sense that it sends weak equivalences into isomorphisms, and the \mbox{precomposition}
$$
(\pi_0q)^* = i^* \pi_0^* :Hom(\pi_0\cc{H}o(\sr{C}_{fc}), \mathscr{X})\longrightarrow Hom(\mathscr{ C}, \mathscr{X})_+
$$
is an equivalence of categories, for every category $\mathscr{X}$.
\cqd
\end{theorem}

\begin{appendix}

\section{Without functorial factorization} 

In this appendix we assume that the reader is familiar with the concept of pseudofunctor. The 2-localization obtained assuming functorial factorization differs with the one we will describe in this appendix using the original Quillen's factorization axiom. The homotopy category will be $\cc{H}o_{fc}(\sr{C})$, not $\cc{H}o(\sr{C}_{fc})$, which is a different 2-category in the absence of functorial factorization, see \ref{listo}. The localising arrow is a pseudofunctor, not a 2-functor, and more fundamentally,  there are fibrant and cofibrant replacement on the homotopy 2-category, which allow to obtain the fibrant-cofibrant replacements as a composition, while in the treatment with functorial factorization there is no fibrant or cofibrant replacement 2-functors on the homotopy 2-category, only a fibrant-cofibrant simultaneous replacement.

\vspace{1ex} 

We refer to the definition \ref{cofibrantreplacement} and consider the following diagrams:
\begin{equation} \label{pseudoestructura}
\xymatrix@C=7ex@R=5ex
   {
     QX \ar@<1.2ex>[r]^{Q(id_X)}
        \ar@{}| {} 
        \ar@<-1.2ex>[r]^{id_{QX}}
        \ar[d]^{p_X}
   & QX \ar@{->>}[d]^{p_X}|{\circ} 
  \\
     X \ar@<1.2ex>[r]^{id_X}
        \ar@{}[r]|{} 
        \ar@<-1.2ex>[r]^{id_X}
   & X
   }
\hspace{5ex}  
\xymatrix@C=6ex@R=5ex
  {
    QX \ar[r]^{Qf}
       \ar@{->>}[d]^{p_X}|\circ
       \ar@/^1.5pc/[rr]^{Q(gf)}
  & QY \ar[r]^{Qg}
       \ar@{->>}[d]^{p_Y}|\circ
  & QZ \ar@{->>}[d]^{p_Z}|\circ
 \\ X \ar[r]^f
      \ar@/_1.5pc/[rr]^{gf}
  & Y \ar[r]^g
  & Z
  }
\end{equation}
Clearly, the functoriality of $Q$ would follow if $p_X$ and $p_Z$ were monomorphisms, but this is not the case. However, in the homotopy category the equalities in the lower part of the diagrams actually lift to homotopies $id_{QX} \Hpy{} Q(id_X)$, 
$Qg \, Qf \Hpy{} Q(gf)$ in the upper part, yielding  2-cells that determine a \emph{pseudofunctor structure} for $Q$. Moreover, it extends to a pseudofunctor defined in the homotopy 2-category, fact which is necessary to perform the composition with the dual fibrant replacement pseudofunctor. 

\vspace{1ex}

 We briefly describe now how this is done, complete definitions and proofs are developed in \cite{e.d.} for model bicategories. The particular case of model categories merits to be treated apart because it is much simpler, avoiding the difficulties and complications necessary to be able to deal with non invertible 2-cells, which actually give rise to a different theory.

 
\vspace{1ex}

Without functorial factorzation pseudofunctors become unavoidable, and we need Theorem \ref{hom_iso} for the hom 2-categories whose objects are now pseudofunctors. 
It is immediate to observe that the construction of $\HAo$,
\ref{hpy_identity}, \ref{horiz_composition_def}, as well as the proofs in \ref{pu_i}, \ref{key}, \ref{prop_key} and \ref{categorical_2aspect}, can be done literally word by word by requiring that the cylinders (Definition   \ref{cylindrical_homotopias}), have the arrow $W \mr{s} Z$ in a subcategory
$\cc{F} \subset \Sigma$ containing the identities.
A 2-category \mbox{$\HAof \subset \HAo$} is thus determined. 
%
%
%
%
On the other hand, Theorem \ref{hom_iso} generalizes (but not easily) to the hom 2-category $pHom_p$ which has pseudofunctors as objects, for a proof we refer the reader to \mbox{\cite[3.38, 3.40]{e.d.2} }.  We have the following generalisation of Theorem \ref{hom_iso}:
\begin{theorem} \label{i_iso_for_pHom}
The inclusion 2-functor $i: \sr{A} \mr{} \HAof$  induces a 
\mbox{2-category} isomorphism
$i^*: pHom_{p}(\HAof, \mathscr{D})_+\longrightarrow pHom_{p}(\mathscr{A}, \mathscr{D})_+$. \cqd
\end{theorem}
Note that regardless of $\cc{F}$, the subscript "$+$" \, still indicates the 2-functors that send the whole class
$\Sigma$ into equivalences.

\vspace{1ex}

We set $\sr{A} = \sr{C}$ and let
$\cc{F}$ be the class of trivial fibrations so that  
$\Hof$ is the 2-category with 2-cells determined by the fibrant (left) homotopies.

\vspace{1ex}

The concepts dual to that of cylinder and homotopy are also necessary, and fibrant homotopies and their duals now play an essential role. It is necessary to establish several facts about them, which we do now.

\subsection{Right homotopies}

Path objects and right homotopies are cylinders and homotopies in the dual model category. As with left homotopies, we get a more general version of Quillen's right homotopies. A homotopy 2-category can also be obtained by taking as 2-cells the classes of finite sequences of right homotopies. We say that two right homotopies $K$ and $K'$ are in the same class if and only if $\widehat{FK}=\widehat{FK'}$ for every 2-functor $F: \mathscr{ C} \longrightarrow \mathscr{D}$ such that $F(\mathcal{W}) \subseteq Equiv(\mathscr{D})$, for every 2-category $\mathscr{D}$. We define the vertical and horizontal compositions as we did before, thus building a 2-category that we call $\Ho^r$ to distinguish it from
$\Ho$ which in this section we will denote $\Ho^l$.

Considering the inclusion $i: \mathscr{C} \longrightarrow \Ho ^r$, the precomposition also gives us an isomorphism as in Theorem \ref{hom_iso}, but this does not mean that both categories are isomorphic, since the 2-functor $i$ in neither of the two cases sends weak equivalences into equivalences.

On the other hand, a right homotopy $K=(P, k)$ and a left homotopy $H=(C, h)$ are related if $\widehat{FK}=\widehat{FH}$ for all $F: \mathscr{C} \longrightarrow \mathscr{D}$ that sends weak equivalences into equivalences. Although $\Ho ^r$ and $\Ho ^l$ do not have to coincide, the following result will allow us to establish that both categories coincide when the objects are fibrant and cofibrant, $\Hofc ^r = \Hofc ^ l$.

Now we explicitly precise  this concepts and considerations in the case of
$q$-homotopies:

\begin{definition} \label{q-path}
A q-path object $P=(V, \delta_0, \delta_1, \sigma)$ for an object $Y$ is a factorization of the diagonal
            \begin{center}
                $\xymatrix{
                    Y \ar@/^2pc/[rrrr]^{\Delta_Y} \ar[rr]^{\sigma} &&V \ar[rr]^(.4){(\delta_0, \delta_1)}  && Y\times Y,
                }$
            \end{center}
    with $(\delta_0, \delta_1)$ a fibration and $\sigma$ a weak equivalence. When $\sigma$ is a cofibration, we say that the path object is \emph{cofibrant}.
     By axiom M4 every object $Y$ has at least one cofibrant  path-object.
\end{definition}

\begin{definition} \label{q-h-right}
A right q-homotop\'y $K: \xymatrix{f \ar@2{~>}[r] & g}$ with $q$-path-object \mbox{$P=(V, \delta_0, \delta_1 , \sigma)$} (for $Y$) is a morphism $k: X \longrightarrow V$ satisfying $\delta_0k=f$ and $\delta_1k=g.$
    \begin{center}
            $\xymatrix@R=2ex{
                    X \ar[r]^{k} & V \ar@{->>}[rr]^{(\delta_0, \delta_1)} && Y \times Y\\
                    && Y \ar[ul]^{\sigma}| \circ \ar[ur]_{\Delta_Y}
            }$,
    \end{center}
\end{definition}

\begin{proposition} \label{lefttoright}
    Let $f,g: X \longrightarrow Y$, $H = (C, h): \xymatrix{ f \ar@2{~>}[r] & g }$ be a \mbox{$q$-homotopy} with cylinder $C=(W, Z, d_0, d_1, s, x)$, and let $P=(V, \delta_0, \delta_1, \sigma)$ be a path-object of $Y$ (which we can choose cofibrant). If $X$ is cofibrant, then there is a right $q$-homotopy $K: \xymatrix{f \ar@2{~>}[r] & g }$ with path-object $P$ such that $[ K]=[H]$.
\end{proposition}

\begin{proof}
    $H$ is of the form $\xymatrix@R=2ex{ X \ar@<3pt>[rr]^{d_0}
                                            \ar@<-3pt>[rr]_{d_1}
                                            \ar[dr]_{id} && W
                                            \ar[rr]^{h} \ar[dl]^{s} && Y, \\
                                            & X \\ }$
   and by Lemma \ref{lemma2_quillen} both $d_0$ and $d_1$ are trivial cofibrations, so there is a morphism $k': W \longrightarrow V$ that making commutative the following diagram
$$
\xymatrix@C=4ex@R=2ex
      {
       X \ar[rr]^{\sigma f}
         \ar@{ >->}[dd]_{d_0}|\circ
    && V \ar@{->>}[dd]^{(\delta_0, \delta_1)}
    \\\\
       W \ar@{-->}[uurr]^{\exists \hspace{.5mm} k'}
         \ar[rr]^{(fs, h)}
    && Y \times Y
    \\
      }
$$
    
   \noindent Defining $k=k'd_1$, we obtain the right $q$-homotopy $K$ we were looking for:
    \begin{center}
        $\xymatrix@R=2ex{
        X \ar[rr]^{k} && V \ar@<3pt>[rr]^{\delta_0} \ar@<-3pt>[rr]_{\delta_1} && Y. \\
        &&& Y \ar[ul]^{\sigma}|\circ \ar[ur]_{id}\\
        }$
    \end{center}
    
    Given $F: \mathscr{C} \longrightarrow \mathscr{D}$, it is not immediate but also not difficult to prove that $\widehat{FH}=\widehat{FK}$.
    This then tells us that $[K] =  [H]$, concluding the proof.
\end{proof}

Note that for fibrant-cofibrant objects any fibrant cylinder or cofibrant path object is necessarily inside $\sr{C}_{fc}$, it follows:
\begin{remark} \label{lefttorightbis}
In \ref{lefttoright} when $X$ and $Y$ are fibrant-cofibrant all the objects involved are fibrant-cofibrant.  \cqd
\end{remark}
From the previous proposition together with its dual version, it follows that when the objects are fibrant and cofibrant, the classes of right \mbox{$q$-homotopies} correspond to the classes of left $q$-homotopies,
 so that $\fcHo$  are the same 2-category in either version. Considering its remark, the same holds for $\Hofc$.


Furthermore we obtain for free that the hom-categories of the 2-categories 
$\fcHo$ and $\Hofc$ are locally small. 
\begin{corollary} \label{loc_small}
    Let $X$, $Y$ be fibrant and cofibrant, and let $f, g: X \longrightarrow Y$ be morphisms in $\mathscr{C}$. Then $\fcHo[X,Y][f, g]$ and $\Hofc[X,Y][f, g]$ are sets.
\end{corollary}

\begin{proof}
Fix any cylinder $C$, by applying the proposition and its dual one after the other it follows that for all the 2-cells $[H]: f \Mr{} g$ we can choose a homotopy $H$ with $C$ as its cylinder.
\end{proof}
\begin{comentario} \label{germcoment2}
Considering Lemma \ref{lemma_3_steps} we assume in Comment \ref{germcoment1} that the homotopies are q-homotopies. Quillen also defines an equivalence relation between homotopies (see \cite{Qui} Ch.I $\S$2), and we have verified that the relation between homotopies defined by Quillen is just our germ relation under a different formulation. He also defines a "correspondence" relation between left homotopies $H$ and right homotopies $K$, $H \crpd K$. It can be easily seen that if $H \crpd K$ then $[H] = [K]$ (compare with Lemma 
\ref{germ==>adhoc}). On the other hand by its very definition the right homotopy $K$ constructed in Proposition \ref{lefttoright} is such that $H \crpd K$. In this way Lemma 1 in \cite{Qui} Ch.I $\S$2 corresponds to our \mbox{Proposition \ref{lefttoright}.} Quillen shows that classes of homotopies can be composed vertically and horizontally, but it does not mention the compatibility between both compositions, which suggests that this compatibility would not be valid. See more on this relation between homotopies in an appendix in \cite{e.d.2}.
\end{comentario}


\subsection{Some properties on fibrant homotopies}

By inspecting the proof of item 2) in Lemma \ref{lemma_3_steps} it is clear that we also have a proof of the following:
\begin{proposition} \label{fibrante}
Given any two objects $X$, $Y$, $f, g: X \longrightarrow Y$, and $H$  a fibrant homotopy from $f$ to 
$g$, there exists a fibrant $q$-homotopy $H'$ such that  
\mbox{$[H]=[H']$.} 
\end{proposition}

 \begin{proposition} \label{lift2cells}
  Let $X \mr{p} Y$ be a trivial fibration, and $p\,f \Mr{[H]} p\,g$ be a 2-cell in $\Ho(Z, Y)$ with $H$ a $q$-homotopy,
$H = (C,\,h)$. Then there exists a unique 2-cell $f \Mr{[H']} g$  
in $\Ho(Z, Y)$ with $H'$ a q-homotopy (with same cylinder as $H$) such that 
$[H] = p\,[H']$. Note that if $H$ is fibrant, so is $H'$. 
\end{proposition}
\begin{proof}
$$
\xymatrix@C=12ex@R=7ex
   {
    Z \amalg Z \ar[r]^{\binom{f}{g}}
               \ar@{>->}[d]_{\binom{d_0}{d_1}} 
  & X \ar@{->>}[d]^{p}|\circ 
 \\ W \ar@{-->}[ur]^{\exists \,h'}
      \ar[r]^{h} 
  & Y
   }{}
$$
That $[H] = p\,[H']$ follows immediately by \ref{rcompoq}, see (\ref{dfeq}) in  Definition \ref{horiz_composition_def}. For the uniqueness, if
$p\,[H] = p\,[K]$ then $Fp\, \widehat{FH} =  Fp\, \widehat{FK}$, and since $Fp$ is an equivalence, it follows $\widehat{FH} =  \widehat{FK}$, thus 
$[H] = [K]$.
\end{proof}
The need of the fibrant homotopies is precisely to allow the use of Proposition \ref{lift2cells}, for which the 2-cells considered have to be determined by \mbox{$q$-homotopies,} which by Proposition \ref{fibrante} is the case for fibrant homotopies,  but not for the general 2-cell of 
$\Ho$. 

\subsection{Fibrant and Cofibrant replacement pseudofunctors.}

We refer to Definition \ref{cofibrantreplacement} and note that the definition of $Q$ on objects and arrows does not use the functoriality of the factorization. In fact, for each $X$ in $\mathscr{C}$, we factor $0 \longrightarrow X$ obtaining a cofibrant object $QX$ and a trivial fibration $p_X: QX \longrightarrow X$. If $X$ is already cofibrant, we choose  $QX = X$ and $p_X = id_X$. Given $f:X \longrightarrow Y$, $Qf:QX \longrightarrow QY$ is defined satisfying the equation $p_{Y}Qf=fp_{X}$ as we see in the diagram
\begin{center}
   $\xymatrix@R=2.5ex
     {
      0 \ar@{ >->}[rr] \ar@{ >->}[dd] && QY \ar@{->>}[dd]^{p_Y}|\circ 
   \\
   \\
      QX \ar@{-->}[uurr]|{Qf} \ar@{->>}[r]_{p_X}|\circ & X \ar[r]_{f} & Y
   \\
      }$
\end{center}
As before, by axiom M5, if $f$ is a weak equivalence, then so is $Qf$.

It remains then to define $Q$ on the homotopies and to define the pseudofunctor structure.

\vspace{1ex}

We found convenient to say that a 2-cell determined by a fibrant homotopy is a \emph{fibrant 2-cell}. 

\vspace{1ex}

1. Definition of $Q$ on a fibrant $2$-cell $[H]$,
$$
\xymatrix@R=2.5ex@C=7ex
    {
  \\ 
     \text{Given \hspace{-7ex}} 
   &
     X  \ar@<1.5ex>[r]^{f}
        \ar@{}[r]|{\Downarrow[H]} 
        \ar@<-1.5ex>[r]_{g}
   & Y
   & \text{\hspace{-7ex} we have a diagram }
     }
\xymatrix@C=7ex@R=5ex
   {
     QX \ar@<1.2ex>[r]^{Qf}
        \ar@{}| {} 
        \ar@<-1.2ex>[r]^{Qg}
        \ar[d]^{p_X}
   & QY \ar@{->>}[d]^{p_Y}|{\circ} 
  \\
     X \ar@<1.5ex>[r]^{f}
        \ar@{}[r]|{\Downarrow[H]} 
        \ar@<-1.5ex>[r]_{g}
   & Y
   }  
$$
$[H]\,p_X$ is a fibrant 2-cell and so by Proposition 
\ref{lift2cells} there is a unique $[H']$ such that $p_Y\,[H'] = [H]\,p_X$. We set $Q[H] = [H']$. By this definition $Q[H]$ is a fibrant 2-cell characterised by the equation
\begin{equation} \label{char1}
p_Y\,Q[H] =  [H]\,p_X.  
\end{equation}

2. Definition of the pseudofunctor structure for $Q$,

Similarly, considering the equality in the lower part of the diagrams in   \ref{pseudoestructura} as fibrant 2-cells, we obtain fibrant
2-cells $id_{QX} \Mr{\xi_X} Q(id_X)$, 
\mbox{$Qg\,Qf \Mr{\phi_{gf}} Q(g\,f)$} characterised by the equations,
\begin{equation} \label{char2-3}
p_X \, \xi_X = p_X, \hspace{4ex} p_Z\,\phi_{gf} = g\,f\,p_X
\end{equation} 
Using the characterisations \ref{char1}, \ref{char2-3} it readily follow all the equations required for the pseudofunctor axioms. \cqd

\vspace{1ex}

In what follows we will work with both left fibrant and right cofibrant  homotopies, for clarity we will indicate with a superscript 
"$\ell\,f$", "$r\, c$" 
which are the 2-cells of the homotopy 2-category being considered.

From 1. and 2. above we have:
\begin{proposition} {\bf (Cofibrant replacement)} \label{cofrepl}
There is a pseudofunctor 
$\Holf \mr{Q} \Holf$ 
together with a 2-natural transformation $Q \Mr{p} id$, such that for all $X$, $QX$ is cofibrant and $p_X$ is a trivial fibration. If $X$ is already cofibrant, {$QX = X$ and 
$p_X = id_X$.}
\end{proposition}
The dual statement takes the form:
\begin{proposition} {\bf (Fibrant replacement)} \label{fibrepl}
There is a pseudofunctor $\Horc \mr{R} \Horc$ together with a 2-natural transformation $id \Mr{v} R$, such that for all $X$, $RX$ is fibrant and $v_X$ is a trivial cofibration. If $X$ is already fibrant, $RX = X$ and 
$v_X = id_X$.
\end{proposition}

{\bf Composition of the fibrant with the cofibrant replacements.} We can apply the fibrant replacement after the cofibrant one and obtain a fibrant-cofibrant replacement as follows:

Let $X \mrdos{f}{g} Y$, $f \Mr{[H]} g$ be a 2-cell in $\Holf$, then 
$Qf \Mr{Q[H]} Qg$ has cofibrant domain $QX$, thus by Proposition \ref{lefttoright} it can be considered as a 2-cell in $\Horc$, thus we can apply the pseudofunctor $R$, obtaining $RQf \Mr{RQ[H]} RQg$, which has a fibrant codomain $RQY$, and in turn by the dual of \mbox{Proposition \ref{lefttoright}} it can be reconsidered as a 2-cell in $\Holf$. In this way the pseudofunctors  $Q$ and $R$ determine by composition  a pseudofunctor 
$\Holf \mr{} \Holf$. We have:
\begin{proposition} [{\bf Fibrant-cofibrant replacement}]
\label{fib-cofib}
There exist a pseudofunctor 
\mbox{$\Holf \mr{RQ} \Holf$} which is the composition of $R$ after $Q$,
and 2-natural transformations 
\mbox{
      $
      id \Ml{p} Q
\xymatrix@C=5ex
    {
     {} \ar@2{->}[r]^{v\, Q}
   & {}
    }
      R\,Q
     $ 
   } 
such that $p_X$ is a trivial fibration and  
$(v\,Q)_X$ is a trivial cofibration  for each $X$.
All the $(R Q)X$ are fibrant-cofibrant objects, and if $X$ is already fibrant-cofibrant, then \mbox{$(RQ)X = X$,} $p_X = id_X$ and $(vQ)_X = id_{QX}$.
If $f$ is a weak equivalence so is $(RQ)f$. \cqd
\end{proposition}


\subsection{The 2-localization theorem.} Now we can return safely to the previous simpler notation, we drop the superscript $\ell$ assuming the homotopies are left homotopies, $\Ho = \Hol$, also by Lemma \ref{lemma_3_steps}, and Proposition \ref{lefttoright} and its dual, for fibrant-cofibrant objects, we have
$\fcHo = \fclfHo = \fcrcHo$.

\vspace{1ex}
    
Let $j$ be the full inclusion \mbox{$\fcHo \mr{j} \Hof$}, it is clear that $RQ$ factors 
\xymatrix@C=5ex
   {
    \Hof \ar[r]^{r}
    \ar@{->}@<-4pt>`u[rr]`[rr]^{RQ}[rr]
  & \fcHo \ar[r]^j  
  & \Hof
   }, 
   \mbox{$ RQ = jr$.} 
We define $q$ as the composition
$\sr{C} \mr{i} \Hof \mr{r} \fcHo$.
As in the subsection \ref{2location} we now want to see that the pseudofunctor \mbox{$q:\mathscr{C} \longrightarrow \fcHo$} determines the 2-localization of $\sr{C}$ at the weak equivalences, that is, that we have a 2-category pseudoequivalence
$$
q^*: pHom_p(\fcHo, \mathscr{D}) \mr{}
             pHom_{p}(\mathscr{C}, \mathscr{D})_+
$$
for every 2-category $\mathscr{D}$.

\vspace{1ex}

Since $q = r \, i$, to prove that $q^\ast$ is a pseudoequivalence it will be enough to see that 
$r^*: pHom_p(\fcHo, \sr{D}) \longrightarrow pHom_{p}(\cc{H}o^f(\sr{C}), \mathscr{D})_+$
and \mbox{$i^*: pHom_p(\cc{H}o^f(\sr{C}), \sr{D})_+ \longrightarrow pHom_{p}(\mathscr{C}, \mathscr {D})_+$}
are both 2-category pseudoequivalences. Concerning 
$\,i^*$, Theorem \ref{i_iso_for_pHom} already tells us that it is indeed an isomorphism.
Having at our disposal  the pseudofunctors $Q$ and $R$, the proof of the theorem corresponding to Theorem \ref{pseudoequivalence} is much simpler.
\begin{theorem} \label{pseudoequivalence_p}
The 2-functors
$$
\xymatrix
    {
     pHom_p(\fcHo, \sr{D})
                     \ar@<5pt>[rr]^{r^*}
                     \ar@<-5pt>@{<-}[rr]_{j^*}
   &&
     pHom_{p}(\cc{H}o^f(\sr{C})
     , \mathscr{D})_+
   }
$$
determine a 2-category pseudoequivalence.
\end{theorem}
\begin{proof}
We already have $j^*\, r^* = (r\,j)^* = id^* = id$, then we only have to see
$id \simeq r^*\,j^* {= (jr)^* = (RQ)^*}$.
Let $id \Ml{p^*} Q^*
\xymatrix@C=6ex{{} \ar@2{->}[r]^{{(vQ)}^*}&{}} {(R Q)}^*$ the pseudo-natural transformations induced by those of the Proposition \ref{fib-cofib}. To see that they induce the sought equivalence, Proposition \ref{prop_infinite} tells us that it is enough to see that for each $\cc{H}o^f(\sr{C}) \mr{F} \mathscr{D}$ that sends weak equivalences into equivalences, and for every $X$ in $\sr{C}$, it suffices to show that
$((p^*)_F)_X = F(p_X)$ and $(({(vQ)}^*)_F)_X = F(v_{QX})$ are equivalences, which follows precisely from the Propositions \ref{cofrepl} and \ref{fibrepl} and the hypothesis made on $F$.
\end{proof}

\end{appendix}

\end{document}